\documentclass[12pt]{article} 
\usepackage{amsfonts,amsmath,latexsym,amssymb,mathrsfs,amsthm,comment,setspace}

\let\OLDthebibliography\thebibliography
\renewcommand\thebibliography[1]{
	\OLDthebibliography{#1}
	\setlength{\parskip}{0pt}
	\setlength{\itemsep}{0pt plus 0.0ex}
}

\def\numberlikeadb{\global\def\theequation{\thesection.\arabic{equation}}}
\numberlikeadb
\newtheorem{theorem}{Theorem}[section]
\newtheorem{lemma}[theorem]{Lemma}
\newtheorem{corollary}[theorem]{Corollary}

\newtheorem{proposition}[theorem]{Proposition}
\newtheorem{remark}[theorem]{Remark}

\usepackage{color} 

\allowdisplaybreaks
\usepackage{lscape}
\usepackage{caption}
\usepackage{multirow}
\begin{document}

	\title{Improved bounds in Stein's method for functions of multivariate normal random vectors}
	\author{Robert E. Gaunt\footnote{Department of Mathematics, The University of Manchester, Oxford Road, Manchester M13 9PL, UK, robert.gaunt@manchester.ac.uk; heather.sutcliffe@manchester.ac.uk}\:\, and Heather Sutcliff$\mathrm{e}^{*}$}

	\date{} 
	\maketitle
	
	\vspace{-5mm}
	
	
\begin{abstract}In a recent paper, Gaunt \cite{gaunt normal} extended Stein's method to limit distributions that can be represented as a function $g:\mathbb{R}^d\rightarrow\mathbb{R}$ of a centered multivariate normal random vector $\Sigma^{1/2}\mathbf{Z}$ with $\mathbf{Z}$ a standard $d$-dimensional multivariate normal random vector and $\Sigma$ a non-negative definite covariance matrix. In this paper, we obtain improved bounds, in the sense of weaker moment conditions, smaller constants and simpler forms, for the case that $g$ has derivatives with polynomial growth. We obtain new non-uniform bounds for the derivatives of the solution of the Stein equation and use these inequalities to obtain general bounds
on the distance, measured using smooth test functions, between the distributions of $g(\mathbf{W}_n)$ and $g(\mathbf{Z})$, where $\mathbf{W}_n$ is a standardised sum of random vectors with independent components and $\mathbf{Z}$ is a standard $d$-dimensional multivariate normal random vector. We apply these general bounds to obtain bounds for the chi-square approximation of the family of power divergence statistics (special cases include the Pearson and likelihood ratio statistics), for the case of two cell classifications, that improve on existing results in the literature.
\end{abstract}

	
	\noindent{{{\bf{AMS 2010 Subject Classification:}}} Primary 60F05; 62E17

\section{Introduction}

Let $\mathbf{Z}$ be a standard $d$-dimensional multivariate normal random vector and let $\Sigma\in\mathbb{R}^{d\times d}$ be a non-negative definite covariance matrix, so that $\Sigma^{1/2}\mathbf{Z}\sim \mathrm{MVN}_d(\mathbf{0},\Sigma)$. By the continuous mapping theorem, if a sequence of $d$-dimensional random vectors  $(\mathbf{W}_n)_{n\geq1}$ converges in distribution to $\Sigma^{1/2}\mathbf{Z}$, then, for any continuous function $g:\mathbb{R}^d\rightarrow\mathbb{R}$, $(g(\mathbf{W}_n))_{n\geq1}$ converges in distribution to $g(\Sigma^{1/2}\mathbf{Z})$. In a recent work, \cite{gaunt normal} developed Stein's method \cite{stein} for the problem of obtaining explicit bounds on the distance between the distributions of $g(\mathbf{W}_n)$ and $g(\Sigma^{1/2}\mathbf{Z})$, measured using smooth test functions. 
Henceforth, for ease of notation, we drop the subscript from $\mathbf{W}_n$. 

The basic approach used by \cite{gaunt normal} (a version of which was first used for chi-square approximation by \cite{pickett,reinert 0}) is as follows. Consider the multivariate normal Stein equation \cite{barbour2,goldstein1,gotze} with test function $h(g(\cdot))$:
\begin{equation} \label{mvng} \nabla^\intercal\Sigma\nabla f(\mathbf{w})-\mathbf{w}^\intercal\nabla f(\mathbf{w})=h(g(\mathbf{w}))-\mathbb{E}[h(g(\Sigma^{1/2}\mathbf{Z}))],
\end{equation} 
which has solution
\begin{equation}\label{mvnsolnh}f_h(\mathbf{w})=-\int_{0}^{1}\frac{1}{t}\big\{\mathbb{E}[h(g(t\mathbf{w}+\sqrt{1-t^2}\Sigma^{1/2}\mathbf{Z}))]-\mathbb{E}[h(g(\Sigma^{1/2}\mathbf{Z}))]\big\}\,\mathrm{d}t.
\end{equation}
In the univariate case, with $d=1$ and $\Sigma=1$, the multivariate normal Stein equation (\ref{mvng}) reduces to the standard normal Stein equation \cite{stein}
\begin{equation*}f''(w)-wf'(w)=h(g(w))-\mathbb{E}[h(g(Z))],
\end{equation*}
where $Z\sim N(0,1)$, and the solution is given by
\begin{align}\label{n01}f_h'(w)&=-\mathrm{e}^{w^2/2}\int_w^\infty\{h(t)-\mathbb{E}[h(g(Z))]\}\mathrm{e}^{-t^2/2}\,\mathrm{d}t \\
\label{n02}&=\mathrm{e}^{w^2/2}\int_{-\infty}^w\{h(t)-\mathbb{E}[h(g(Z))]\}\mathrm{e}^{-t^2/2}\,\mathrm{d}t.
\end{align}
The quantity of interest $|\mathbb{E}[h(g(\mathbf{W}))]-\mathbb{E}[h(g(\Sigma^{1/2}\mathbf{Z}))]|$ can now be bounded by bounding the expectation 
\begin{equation}\label{emvn}\mathbb{E}[\nabla^\intercal\Sigma\nabla f_h(\mathbf{W})-\mathbf{W}^\intercal\nabla f_h(\mathbf{W})].
\end{equation}
Taking the supremum of (\ref{emvn}) over all $h$ in some measure determining class of test functions $\mathcal{H}$ yields a bounds on the distance between $g(\mathbf{W})$ and $g(\Sigma^{1/2}\mathbf{Z})$ as measured by the integral probability metric $d_{\mathcal{H}}(\mathbf{W},\Sigma^{1/2}\mathbf{Z}):=\sup_{h\in\mathcal{H}}|\mathbb{E}[h(g(\mathbf{W}))]-\mathbb{E}[h(g(\Sigma^{1/2}\mathbf{Z}))]|$. In Stein's method, the following classes of test functions are often used: 
\begin{align*}\mathcal{H}_{\mathrm{K}}&=\{\mathbf{1}(\cdot\leq z\,:\,z\in\mathbb{R}\}, \\
\mathcal{H}_{\mathrm{W}}&=\{h:\mathbb{R}\rightarrow\mathbb{R}\,:\,\text{$h$ is Lipschitz with $\|h'\|\leq1$}\}, \\
\mathcal{H}_p&=\{h:\mathbb{R}\rightarrow\mathbb{R}\,:\, \text{$h^{(p-1)}$ is Lipschitz with $\|h^{(k)}\|\leq1$, $1\leq k\leq p$}\},
\end{align*}
which induce the Kolmogorov, Wasserstein and smooth Wasserstein ($p\geq1$) distances, denoted by $d_{\mathrm{K}}$, $d_{\mathrm{W}}$ and $d_p$, respectively. As discussed by \cite{bh85}, in theoretical settings the $d_p$ distance is a natural probability metric to work, particularly in the context of quantitative limit theorems with faster convergence rates than the $O(n^{-1/2})$ Berry-Esseen rate.
Here, and throughout the paper, $\|\cdot\|:=\|\cdot\|_\infty$ is the usual supremum norm of a real-valued function. Note that $d_1=d_{\mathrm{W}}$.

One of the main contributions of \cite{gaunt normal} was to obtain suitable bounds for the solution (\ref{mvnsolnh}) of the Stein equation  (\ref{mvng}) that hold for a large class of functions $g:\mathbb{R}^d\rightarrow\mathbb{R}$ (this is necessary in order to obtain good bounds on the quantity (\ref{emvn})). In particular, \cite{gaunt normal} obtained bounds for the case that the derivatives of $g$ have polynomial growth (this covers, for example, chi-square approximation: $g(w)=w^2$). These bounds were used by \cite{gaunt normal} to derive explicit bounds on the distance between the distributions of $g(\mathbf{W})$ and $g(\mathbf{Z})$ for the 
case that $\mathbf{W}$ is a sum of independent random vectors with independent components, that is $\mathbf{W}=(W_1,\ldots,W_d)^\intercal$, where, for $j=1,\ldots,d$, $W_j=n_j^{-1/2}\sum_{i=1}^{n_j}X_{ij}$, and the $X_{ij}$ are independent random variables with zero mean and unit variance. Notably, \cite{gaunt normal} obtained bounds with faster rates of convergence than the $O(n^{-1/2})$ Berry-Esseen rate under additional matching moments between the $X_{ij}$ and the standard normal distribution, and when $g$ is an even function $(g(\mathbf{w})=g(-\mathbf{w})$ for all $\mathbf{w}\in\mathbb{R}^d$).

The aforementioned results of \cite{gaunt normal} seem to have broad applicability, in part because many distributional approximations in probability and statistics assess the distance between the distributions of random variables that can be expressed in the form $g(\mathbf{W})$ and $g(\Sigma^{1/2}\mathbf{Z})$, where $\mathbf{W}$ is close in distribution to $\Sigma^{1/2}\mathbf{Z}$. Indeed, applications include bounds for the chi-square approximation of the likelihood ratio statistic \cite{ar20}, the family of power divergence statistics \cite{gaunt pd} and Friedman's statistic \cite{gr21}, multivariate normal approximation of the maximum likelihood estimator \cite{ag20}, and the bounds for distributional approximation in the delta method \cite{gaunt normal}.

Motivated by the broad applicability of the results of \cite{gaunt normal}, in this paper we improve the results of \cite{gaunt normal} in the form of weaker moment conditions, smaller constants and simpler bounds. Future works (including \cite{fgrs22}) that require results from Stein's method for functions of multivariate normal approximation will reap these benefits. 

In Section \ref{sec2}, we obtain non-uniform bounds on the derivatives of the solution (\ref{mvnsolnh}) of the Stein equation (\ref{mvng}) that have smaller constants than those of \cite{gaunt normal} and improved polynomial growth rate. We achieve these improved bounds through a more focused proof than that used by \cite{gaunt normal} which had derived the bounds for the case of polynomial growth rate $g:\mathbb{R}^d\rightarrow\mathbb{R}$ from a more general framework. We also use the iterative technique of \cite{dgv17} for bounding derivatives of solutions of  Stein equations to obtain bounds on the derivatives of the solution (\ref{mvnsolnh}) in the univariate $d=1$ case (which require weaker differentiability assumptions on $h$ and $g$) that have an optimal $w$-dependence; a crude approach had been used by \cite{gaunt normal} that resulted in bounds with sub-optimal $w$-dependence.

In Section \ref{sec3}, we apply the bounds of Section \ref{sec2} to obtain bounds on the distance between the distributions of $g(\mathbf{W})$ and $g(\mathbf{Z})$, where $\mathbf{W}$ is a sum of independent random vectors with independent components and $\mathbf{Z}$ is a standard $d$-dimensional multivariate normal random vector.  The bounds of Theorem \ref{thm3.2} improve on those of Theorems 3.2--3.5 of \cite{gaunt normal} in terms of smaller constants and weaker moment conditions. In Corollary \ref{simplebds}, we provide simplified bounds without explicit constants.

In Section \ref{sec4}, we provide an application of the general bounds of Section \ref{sec3} to chi-square approximation. We derive explicit bounds for the chi-square approximation of the power divergence family of statistics \cite{cr84} 
 in the case of two cell classifications. The power divergence statistic has a special structure in the case of two cell classifications, which allows us to apply the bounds of Section \ref{sec3}. Our bounds improve on existing results in the literature by holding in stronger probability metrics and having smaller constants, and our Kolmogorov distance bounds have a faster rate of convergence than the only other bounds in the literautre that have the correct dependence on the cell classification probabilities. Moreover, as all our bounds possess an optimal dependence on the cell classification probabilities, we are able to demonstrate the significance of the weaker moment conditions of the bounds of Theorem \ref{thm3.2}, as using the general bounds of \cite{gaunt normal} would lead to bounds with a sub-optimal dependence on these probabilities; 
see Remark \ref{bdfnd}. Finally, some technical lemmas are proved in Appendix \ref{appendix}.

\vspace{3mm}

\noindent{\emph{Notation.}} 
The class $C^n_b(I)$ consists of all functions $h:I\subset\mathbb{R}\rightarrow\mathbb{R}$ for which $h^{(n-1)}$ exists and is absolutely continuous and has bounded derivatives up to the $n$-th order. For a given $P$, the class $C_P^n(\mathbb{R}^d)$ consists of all functions $g:\mathbb{R}^d\rightarrow\mathbb{R}$ such that all $n$-th order partial derivatives of $g$ exist and are such that, for $\mathbf{w}\in\mathbb{R}^d$, 
\begin{equation*}
	\bigg|\frac{\partial ^k}{\prod_{j=1}^{k}\partial w_{i_j}}g(\mathbf{w})\bigg|^{n/k}\leq P(\mathbf{w}), \quad k=1,\ldots,n.
\end{equation*}
We will also consider the weaker class $C_{P,*}^n(\mathbb{R}^d)$, which consists of all functions $g:\mathbb{R}^d\rightarrow\mathbb{R}$ such that all $n$-th order partial derivatives of $g$ exist and are bounded in absolute value by $P(\mathrm{w})$ for all $\mathbf{w}\in\mathbb{R}^d$. 
We will write $h_n=\sum_{k=1}^n{n\brace k}\|h^{(k)}\|$, where ${n\brace k}=(1/k!)\sum_{j=0}^k(-1)^{k-j}\binom{k}{j}j^n$ is a Stirling number of the second kind (see \cite{olver}).  A standard multivariate normal random vector of dimension $d$ will be denoted by $\mathbf{Z}$, and, in the univariate $d=1$ case, $Z$ will denote a standard normal $N(0,1)$ random variable.  Many of our bounds will be expressed in terms of the $r$-th absolute moment of the $N(0,1)$ distribution, which we denote by $\mu_r=2^{r/2}\Gamma((r+1)/2)/\sqrt{\pi}$.

\section{Bounds for the solution of the Stein equation}\label{sec2}

In the following proposition, we provide bounds for the solution (\ref{mvnsolnh}) of the $\mathrm{MVN}_d(\mathbf{0},\Sigma)$ Stein equation (\ref{mvng}) with test function $h(g(\cdot))$, which improve on results of \cite{gaunt normal}.

\begin{proposition}\label{fprop}Let $P(\mathbf{w})=A+B\sum_{i=1}^d|w_i|^{r_i}$, where $r_i\geq 0$, $i=1,\ldots,d$.  Let $\sigma_{i,i}=(\Sigma)_{i,i}$, $i=1,\ldots,d$. Let $f(=f_h)$ denote the solution (\ref{mvnsolnh}).

\vspace{2mm}

\noindent{(i)} Assume that $\Sigma$ is non-negative definite and $h\in C_b^n(\mathbb{R})$ and $g\in C_P^n(\mathbb{R}^d)$ for $n\geq 1$.
 Then, for all $\mathbf{w}\in\mathbb{R}^d$,
		\begin{align}\bigg|\frac{\partial^nf(\mathbf{w})}{\prod_{j=1}^n\partial w_{i_j}}\bigg|&\leq\frac{h_n}{n}\bigg[A+B\sum_{i=1}^d2^{r_i/2}\big(|w_i|^{r_i}+\sigma_{ii}^{r_i/2}\mu_{r_i}\big)\bigg].\label{p2.21}
		\end{align}
		(ii) Now assume that $\Sigma$ is positive definite and $h\in C_b^{n-1}(\mathbb{R})$ and $g\in C_P^{n-1}(\mathbb{R}^d)$ for $n\geq 2$.  Then, for all $\mathbf{w}\in\mathbb{R}^d$,
		\begin{align}
			\bigg|\frac{\partial^{n}f(\mathbf{w})}{\prod_{j=1}^{n}\partial w_{i_j}}\bigg|&\leq \frac{\sqrt{\pi}\Gamma(\frac{n}{2})}{2\Gamma(\frac{n+1}{2})}h_{n-1}\min_{1\leq l\leq d}
			\bigg[A\mathbb{E}|(\Sigma^{-1/2}\mathbf{Z})_l|\nonumber\\
			&\quad+B\sum_{i=1}^d2^{r_i/2}\big(|w_i|^{r_i}\mathbb{E}|(\Sigma^{-1/2}\mathbf{Z})_l|+\mathbb{E}|(\Sigma^{-1/2}\mathbf{Z})_l((\Sigma^{1/2}\mathbf{Z})_i)^{r_i}|\big)\bigg].\label{p2.22}
		\end{align}
		In the case $\Sigma=I_d$, the $d\times d$ identity matrix, we obtain the simplified bound
		\begin{align}	\bigg|\frac{\partial^{n}f(\mathbf{w})}{\prod_{j=1}^{n}\partial w_{i_j}}\bigg|&\leq \frac{\sqrt{\pi}\Gamma(\frac{n}{2})}{2\Gamma(\frac{n+1}{2})}h_{n-1}
			\bigg[A+B\sum_{i=1}^d2^{r_i/2}\big(|w_i|^{r_i}+\mu_{r_i+1}\big)\bigg].\label{p2.23}
		\end{align}
(iii) Finally, consider the case $d=1$ with $\Sigma=1$. Assume that $h\in C_b^{n-2}(\mathbb{R})$ and $g\in C_P^{n-2}(\mathbb{R})$, where $n\geq 3$ and $P(w)=A+B|w|^r$, $r\geq0$.  Then, for all $w\in\mathbb{R}$,
		\begin{align}|f^{(n)}(w)|\leq h_{n-2}\big[\alpha_rA+2^{r/2}B\big(\beta_r|w|^r+\gamma_r\big)\big],\label{p2.24}
		\end{align}
where $(\alpha_r,\beta_r,\gamma_r)=(4,4,2\mu_r)$ if $0\leq r\leq1$, and	$(\alpha_r,\beta_r,\gamma_r)=(r+3,r+5,(r+1)\mu_{r+1})$ if $r>1$.	

\vspace{2mm}

\noindent{(iv)} If $h(w)=w$ for all $w\in\mathbb{R}$, then the inequalities in parts (i), (ii) and (iii) hold for $g$ in the classes $C_{P,*}^n(\mathbb{R}^d)$, $C_{P,*}^{n-1}(\mathbb{R}^d)$ and $C_{P,*}^{n-2}(\mathbb{R})$, respectively. Moreover, under these assumptions, the polynomial growth rate in the bounds (\ref{p2.21})--(\ref{p2.24}) is optimal.
\end{proposition}

In the following proposition, we obtain bounds for $\psi_m$, the solution of the Stein equation
\begin{equation}\label{mvnpsi}
	\nabla^\intercal\Sigma\nabla \psi_m(\mathbf{w})-\mathbf{w}^\intercal\nabla \psi_m(\mathbf{w})=\frac{\partial^m f(\mathbf{w})}{\prod_{j=1}^{m}\partial w_{i_j}},
\end{equation}
where $f(=f_h)$ is the solution (\ref{mvnsolnh}). Again, the bounds improve on results of \cite{gaunt normal}. Here, and throughout this section, we suppress in the notation the dependence of the solution $\psi_m$ on the components with respect to which $f$ has been differentiated; we do this for ease of notation and because our bounds for $\psi_m$ do not depend themselves on which components $f$ has been differentiated with respect to. The Stein equation (\ref{mvnpsi}) arises in the proof of parts (iii) and (iv) of Theorem \ref{thm3.2}, in which faster convergence rates for the distributional approximation of $g(\mathbf{W})$ by $g(\mathbf{Z})$ are achieved when $g:\mathbb{R}^d\rightarrow\mathbb{R}$ is an even function. 

\begin{proposition}\label{psiprop}Let $P(\mathbf{w})=A+B\sum_{i=1}^d|w_i|^{r_i}$, where $r_i\geq 0$, $i=1,\ldots,d$.  Let $\sigma_{i,i}=(\Sigma)_{i,i}$, $i=1,\ldots,d$.  
\vspace{2mm}

\noindent{(i)} Assume that $\Sigma$ is non-negative definite and $h\in C_b^{m+n}(\mathbb{R})$ and $g\in C_P^{m+n}(\mathbb{R}^d)$ for $m,n\geq 1$.  Then, for all $\mathbf{w}\in\mathbb{R}^d$,
	\begin{align}\bigg|\frac{\partial^n\psi_m(\mathbf{w})}{\prod_{j=1}^n\partial w_{i_j}}\bigg|&\leq\frac{h_{m+n}}{n(m+n)}\bigg[A+B\sum_{i=1}^d3^{r_i/2}\big(|w_i|^{r_i}+2\sigma_{ii}^{r_i/2}\mu_{r_i}\big)\bigg].\label{p2.31}
	\end{align}
(ii)	Now assume that $\Sigma$ is positive definite and $h\in C_b^{m+n-2}(\mathbb{R})$ and $g\in C_P^{m+n-2}(\mathbb{R}^d)$ for $m,n\geq 1$ and $m+n\geq 3$.  Then, for all $\mathbf{w}\in\mathbb{R}^d$,
	\begin{align}
		\bigg|\frac{\partial^{n}\psi_{m}(\mathbf{w})}{\prod_{j=1}^{n}\partial w_{i_j}}\bigg|&\leq \frac{\pi\Gamma(\frac{n}{2})\Gamma(\frac{m+n}{2})}{4\Gamma(\frac{n+1}{2})\Gamma(\frac{m+n+1}{2})}h_{m+n-2}\min_{1\leq k,l\leq d}\bigg[A\mathbb{E}|(\Sigma^{-1/2}\mathbf{Z})_k|\mathbb{E}|(\Sigma^{-1/2}\mathbf{Z})_l|\nonumber\\
		&\quad+B\sum_{i=1}^d3^{r_i/2}\big(|w_i|^{r_i}\mathbb{E}|(\Sigma^{-1/2}\mathbf{Z})_l|\mathbb{E}|(\Sigma^{-1/2}\mathbf{Z})_k|\nonumber\\
		&\quad+2\mathbb{E}|(\Sigma^{-1/2}\mathbf{Z})_k|\mathbb{E}|(\Sigma^{-1/2}\mathbf{Z})_l((\Sigma^{1/2}\mathbf{Z})_i)^{r_i}|\big)\bigg].\label{p2.32}
	\end{align}
 In the case $\Sigma=I_d$, we obtain the simplified bound
	\begin{equation}\label{p2.33} \bigg|\frac{\partial^{n}\psi_{m}(\mathbf{w})}{\prod_{j=1}^{n}\partial w_{i_j}}\bigg|\leq \frac{\sqrt{2\pi}\Gamma(\frac{n}{2})\Gamma(\frac{m+n}{2})}{4\Gamma(\frac{n+1}{2})\Gamma(\frac{m+n+1}{2})}h_{m+n-2}\bigg[A+B\sum_{i=1}^d3^{r_i/2}\big(|w_i|^{r_i}+2\mu_{r_i+1}\big)\bigg].
	\end{equation}
(iii)	Finally, consider the case $d=1$ with $\Sigma=1$.  Assume that $h\in C_b^{m-1}(\mathbb{R})$ and $g\in C_P^{m-1}(\mathbb{R})$, where $m\geq 2$ and $P(w)=A+B|w|^r$, $r\geq0$.  Then, for all $w\in\mathbb{R}$,
	\begin{align}|\psi_m^{(3)}(w)|\leq h_{m-1}\big[\tilde\alpha_r A+ 3^{r/2}B\big(\tilde\beta_r|w|^{r}+\tilde\gamma_r\big)\big],\label{p2.34}
	\end{align}
where $(\tilde\alpha_r,\tilde\beta_r,\tilde\gamma_r)=(10,10,10\mu_{r+1})$ if $0\leq r\leq1$, and $(\tilde\alpha_r,\tilde\beta_r,\tilde\gamma_r)=(r^2+r+8,r^2+2r+18,(2r^2+r+5)\mu_{r+1})$ if $r>1$.	

\vspace{2mm}

\noindent{(iv)} If $h(w)=w$ for all $w\in\mathbb{R}$, then the inequalities in parts (i), (ii) and (iii) hold for $g$ in the classes $C_{P,*}^{m+n}(\mathbb{R}^d)$, $C_{P,*}^{m+n-2}(\mathbb{R}^d)$ and $C_{P,*}^{m-1}(\mathbb{R})$, respectively. Moreover, under these assumptions, the polynomial growth rate in the bounds (\ref{p2.31})--(\ref{p2.34}) is optimal.
\end{proposition}

\begin{remark}The following two-sided inequalities are helpful in understanding the behaviour of the bounds (\ref{p2.22}) and (\ref{p2.23}) of Proposition \ref{fprop} and the bounds (\ref{p2.32}) and (\ref{p2.33}) of Proposition \ref{psiprop} for large $m$ and $n$:   
\begin{equation}\label{gamma1}\frac{\sqrt{2}}{\sqrt{n}}<\frac{\Gamma(\frac{n}{2})}{\Gamma(\frac{n+1}{2})}<\frac{\sqrt{2}}{\sqrt{n-1/2}}, \quad n\geq 1,
\end{equation}
and
\[\frac{2}{\sqrt{n(m+n)}}<\frac{\Gamma(\frac{n}{2})\Gamma(\frac{m+n}{2})}{\Gamma(\frac{n+1}{2})\Gamma(\frac{m+n+1}{2})}<\frac{2}{\sqrt{(n-1/2)(m+n-1/2)}}, \quad m,n\geq 1.\]
Here we used the inequalities $\frac{\Gamma(x+1/2)}{\Gamma(x+1)}>(x+1/2)^{-1/2}$ for $x>0$ (see the proof of Corollary 3.4 of \cite{gaunt inequality}) and $\frac{\Gamma(x+1/2)}{\Gamma(x+1)}<(x+1/4)^{-1/2}$ for $x>-1/4$ (see \cite{elezovic}).
\end{remark}

\begin{remark}
1. Inequalities (\ref{p2.21})--(\ref{p2.23}) of Propostion \ref{fprop} and inequalities (\ref{p2.31})--(\ref{p2.33}) of Propostion \ref{psiprop} are sharper than the corresponding bounds given in Corollaries 2.2 and 2.3 of \cite{gaunt normal} through improved dependence on the constants $r_1,\ldots,r_d$ and $m$ and $n$. We remark that further more accurate bounds can be given; for example, a slight modification of the proof of inequality (\ref{p2.21}) leads to the improved bound
\begin{align*}\bigg|\frac{\partial^nf(\mathbf{w})}{\prod_{j=1}^n\partial w_{i_j}}\bigg|&\leq\frac{h_n}{n}\bigg[A+B\sum_{i=1}^dI_{n,r_i}\big(|w_i|^{r_i}+\sigma_{ii}^{r_i/2}\mu_{r_i}\big)\bigg], \quad \mathbf{w}\in\mathbb{R}^d,
		\end{align*}
where $I_{n,r}=n\int_{0}^{1}t^{n-1}(t+\sqrt{1-t^2})^r\,\mathrm{d}t$, and it is clear that $I_{n,r}\leq 2^{r/2}$. This integral cannot in general be expressed in terms of elementary functions, but for given $n$ and $r$ can be evaluated by computational algebra packages. In specific applications in which $n$ and $r_1,\ldots,r_d$ are known, this bound could be applied to improve constants. Similarly, the constant $2^{r_i/2}$ in the bounds (\ref{p2.22}) and (\ref{p2.23}) can be improved by replacing $2^{r_i/2}$ by $J_{n,r_i}=\frac{2\Gamma{((n+1)/2)}}{\sqrt{\pi}\Gamma{(n/2)}}\int_{0}^{1}\frac{t^{n-1}}{\sqrt{1-t^2}}(t+\sqrt{1-t^2})^{r_i}\,\mathrm{d}t$, and the factor $3^{r_i/2}$ in inequality (\ref{p2.31})
 can be improved to $I_{m,n,r_i}=n(m+n)\int_{0}^{1}\int_{0}^{1}t^{m+n-1}s^{n-1}(st+t\sqrt{1-s^2}+\sqrt{1-t^2})^{r_i}\,\mathrm{d}s\,\mathrm{d}t$, whilst the factor $3^{r_i/2}$ in inequalities (\ref{p2.32}) and (\ref{p2.33}) can be improved to $J_{m,n,r_i}=\frac{4\Gamma((n+1)/2)\Gamma((m+n+1)/2)}{\pi\Gamma(n/2)\Gamma((m+n)/2)}\int_{0}^{1}\int_{0}^{1}\frac{t^{m+n-1}}{\sqrt{1-t^2}}\frac{s^{n-1}}{\sqrt{1-s^2}}(st+t\sqrt{1-s^2}+\sqrt{1-t^2})^{r_i}\,\mathrm{d}s\,\mathrm{d}t$. For our purposes, we find it easier to work with the more explicit bounds stated in the propositions, particularly as these inequalities are used to derive inequalities (\ref{p2.24}) and (\ref{p2.34}) leading to a more efficient proof and simpler constants in the final bounds.
 
 \vspace{2mm}
 
\noindent{2.} Inequality (\ref{p2.24}) of Proposition \ref{fprop} and inequality (\ref{p2.34}) of Proposition \ref{psiprop} have a theoretically improved dependence on $r$ than the corresponding bounds of \cite{gaunt normal} in that for large $r$ they are of a smaller asymptotic order; however, for small $r$ they may be numerically larger. More significantly, inequality (\ref{p2.24}) of Proposition \ref{fprop} improves on the corresponding bound of Corollary 2.2 of \cite{gaunt normal} by improving the $w$-dependence of the bound from $|w|^{r+1}$ to $|w|^r$, whilst inequality (\ref{p2.24}) of Proposition \ref{psiprop} improves on the corresponding bound of Corollary 2.3 of \cite{gaunt normal} by improving the $w$-dependence of the bound from $|w|^{r+2}$ to $|w|^r$. This improved $w$-dependence allows us to impose weaker moment conditions in our general bounds in Theorem \ref{thm3.2}. 
\end{remark}

In proving Propositions \ref{fprop} and \ref{psiprop}, we will make use of the following simple lemmas. The proofs are given in Appendix \ref{appendix}.

\begin{lemma}\label{lem2} Suppose $a,b,c,x,y,z\geq0$ and $r\geq0$. Then
	\begin{align}
	\label{abineq}	(ax+by)^r&\leq(a+b)^r(x^r+y^r),\\
	\label{abcineq}	(ax+by+cz)^r&\leq(a+b+c)^r(x^r+y^r+z^r).
	\end{align}
\end{lemma}

\begin{lemma}\label{lem3}Let $T_r(w)=w\mathrm{e}^{w^2/2}\int_{w}^{\infty}t^r\mathrm{e}^{-t^2/2}\,\mathrm{d}t$.

\vspace{2mm}
	
\noindent{1.} Suppose $0\leq r\leq 1$. Then, for $w>0$, 
	\begin{equation}
		T_r(w)\leq w^r.\label{Tr1}
	\end{equation}

\noindent{2.} Suppose $r>1$. Then, for $w>r-1$, 
\begin{equation}
	T_r(w)\leq2w^r.\label{Tr2}
\end{equation}
\end{lemma}

\noindent{\emph{Proof of Proposition \ref{fprop}.}} (i) Suppose that $\Sigma$ is non-negative definite. We first recall inequality (2.2) of \cite{gaunt normal}, which states (under a change of variable in the integral) that, for positive definite $\Sigma$, and $g\in C_P^n(\mathbb{R}^d)$ and $h\in C_b^n(\mathbb{R})$,
\begin{equation}\label{aihp1}
	\bigg|\frac{\partial^nf(\mathbf{w})}{\prod_{j=1}^{n}\partial w_{i_j}}\bigg|\leq h_n\int_{0}^{1}t^{n-1}\mathbb{E}[P(\mathbf{z}_{t,\mathbf{w}}^{\Sigma^{1/2}\mathbf{Z}})]\,\mathrm{d}t,
\end{equation}
where $\mathbf{z}_{t,\mathbf{w}}^{\mathbf{x}}=t\mathbf{w}+\mathbf{x}\sqrt{1-t^2}$. We will denote the $i$-th component of $\mathbf{z}_{t,\mathbf{w}}^{\mathbf{x}}$ by $(\mathbf{z}_{t,\mathbf{w}}^{\mathbf{x}})_i$. 
Using inequality (\ref{aihp1}) with dominating function $P(\mathbf{w})=A+B\sum_{i=1}^{d}|w_i|^{r_i}$ gives the bound
\begin{equation*}
	\bigg|\frac{\partial^nf(\mathbf{w})}{\prod_{j=1}^{n}\partial w_{i_j}}\bigg|\leq h_n\int_{0}^{1}t^{n-1}\mathbb{E}\bigg[A+B\sum_{i=1}^{d}\big|(\mathbf{z}_{t,\mathbf{w}}^{\Sigma^{1/2}\mathbf{Z}})_i\big|^{r_i}\bigg]\,\mathrm{d}t.
\end{equation*}
Now using inequality (\ref{abineq}) of Lemma \ref{lem2}, we have
\begin{equation*}
	|(\mathbf{z}_{t,\mathbf{w}}^{\mathbf{x}})_i|^{r_i}\leq(t+\sqrt{1-t^2})^{r_i}(|w_i|^{r_i}+|x_i|^{r_i})\leq 2^{r_i/2}(|w_i|^{r_i}+|x_i|^{r_i}),
\end{equation*}
where we used basic calculus to bound $t+\sqrt{1-t^2}\leq\sqrt{2}$, $t\in(0,1)$.
Therefore
\begin{equation*}
	\bigg|\frac{\partial^nf(\mathbf{w})}{\prod_{j=1}^{n}\partial w_{i_j}}\bigg|\leq h_n\int_{0}^{1}t^{n-1}\bigg[A+B\sum_{i=1}^{d}2^{r_i/2}(|w_i|^{r_i}+\mathbb{E}[|(\Sigma^{1/2}\mathbf{Z})_i|^{r_i}])\bigg]\,\mathrm{d}t,
\end{equation*}
from which inequality (\ref{p2.21})  follows on evaluating the integral $\int_0^1t^{n-1}\,\mathrm{d}t=1/n$ and using that $(\Sigma^{1/2}\mathbf{Z})_i\sim N(0,\sigma_{ii})$, so that $\mathbb{E}|(\Sigma^{1/2}\mathbf{Z})_i|^{r_i}=\sigma_{ii}^{r_i/2}\mu_{r_i}$.

\vspace{2mm}

\noindent{(ii)} Suppose now that $\Sigma$ is positive definite. Under this assumption, we may recall inequality (2.3) of \cite{gaunt normal}, which states that, for $g\in C_P^n(\mathbb{R}^d)$ and $h\in C_b^n(\mathbb{R})$,
\begin{equation}\label{aihp2}
	\bigg|\frac{\partial^nf(\mathbf{w})}{\prod_{j=1}^{n}\partial w_{i_j}}\bigg|\leq h_{n-1}\min_{1\leq l\leq d}\int_{0}^{1}\frac{t^{n-1}}{\sqrt{1-t^2}}\mathbb{E}\big|(\Sigma^{-1/2}\mathbf{Z})_lP(\mathbf{z}_{t,\mathbf{w}}^{\Sigma^{1/2}\mathbf{Z}})\big|\,\mathrm{d}t.
\end{equation}
With inequality (\ref{aihp2}) at hand, a similar argument to the one used in part (i) of the proof yields the bound
\begin{align*}
	\bigg|\frac{\partial^nf(\mathbf{w})}{\prod_{j=1}^{n}\partial w_{i_j}}\bigg|&\leq h_{n-1}\min_{1\leq l\leq d}\int_{0}^{1}\frac{t^{n-1}}{\sqrt{1-t^2}}\bigg[A\mathbb{E}\big|(\Sigma^{-1/2}\mathbf{Z})_l\big|\\
	&\quad+B\sum_{i=1}^{d}2^{r_i/2}\big(|w_i|^{r_i}\mathbb{E}\big|(\Sigma^{-1/2}\mathbf{Z})_l\big|+\mathbb{E}|(\Sigma^{-1/2}\mathbf{Z})_l((\Sigma^{1/2}\mathbf{Z})_i)^{r_i}\big|\big)\bigg]\,\mathrm{d}t,
\end{align*}
from which inequality (\ref{p2.22}) follows on evaluating $\int_{0}^{1}\frac{t^{n-1}}{\sqrt{1-t^2}}\,\mathrm{d}t=\frac{\sqrt{\pi}\Gamma(n/2)}{2\Gamma((n+1)/2)}$. Inequality (\ref{p2.23}) now follows by setting $\Sigma=I_d$ in  inequality (\ref{p2.22}) and bounding $\mathbb{E}|Z|=\sqrt{2/\pi}<1$.

\vspace{2mm}

\noindent{(iii)} Suppose now that $d=1$ and $\Sigma=1$. We will employ the iterative approach to bounding solutions of Stein equations of \cite{dgv17}. Let $L$ denote the standard normal Stein operator given by $Lf(w)=f''(w)-wf'(w)$. Then, as noted by \cite{dgv17}, an induction on $n$ gives that, for $n\geq3$, 
\begin{equation}\label{DGV}
	Lf^{(n-2)}(w)=(h\circ g)^{(n-2)}(w)+(n-2)f^{(n-2)}(w).
\end{equation}
From (\ref{DGV}), a rearrangement and an application of the triangle inequality gives that 
\begin{equation}
	|f^{(n)}(w)|\leq|wf^{(n-1)}(w)|+(n-2)|f^{(n-2)}(w)|+|(h\circ g)^{(n-2)}(w)|\label{DGVineq}.
\end{equation}
From inequality (\ref{p2.21}) and Lemma 2.1 of \cite{gaunt normal}, we have the bounds
\begin{equation}
	(n-2)|f^{(n-2)}(w)|\leq h_{n-2}\big[A+2^{r/2}B(|w|^r+\mu_r)\big],\label{gn1}
\end{equation}
and
\begin{equation}
	|(h\circ g)^{(n-2)}(w)|\leq h_{n-2}\big[A+B|w|^r\big].\label{gn2}
\end{equation}

It now remains to bound $|wf^{(n-1)}(w)|$, which requires more work to get the correct $w$-dependence. To this end, we obtain from (\ref{n01}), (\ref{n02}) and (\ref{DGV}) the following representations for $f^{(n-1)}(w)$:
\begin{align}
	f^{(n-1)}(w)&=-\mathrm{e}^{w^2/2}\int_{w}^{\infty}[(h\circ g)^{(n-2)}(t)+(n-2)f^{(n-2)}(t)]\mathrm{e}^{-t^2/2}\,\mathrm{d}t\label{DGVsol1}\\
	&=\mathrm{e}^{w^2/2}\int_{-\infty}^{w}[(h\circ g)^{(n-2)}(t)+(n-2)f^{(n-2)}(t)]\mathrm{e}^{-t^2/2}\,\mathrm{d}t.\label{DGVsol2}
\end{align}
 We first consider the case $0\leq r\leq 1$. 
From (\ref{DGVsol1}), we have that, for $w>0$, 
\begin{align}
	|wf^{(n-1)}(w)|&\leq w\mathrm{e}^{w^2/2}\int_{w}^{\infty}\big[|(h\circ g)^{(n-2)}(t)|+(n-2)|f^{(n-2)}(t)|\big]\mathrm{e}^{-t^2/2}\,\mathrm{d}t\nonumber\\
	&\leq h_{n-2}w\mathrm{e}^{w^2/2}\int_{w}^{\infty}\big[2A+B\big((1+2^{r/2})t^r+2^{r/2}\mu_r\big)\big]\mathrm{e}^{-t^2/2}\,\mathrm{d}t\label{DGVineq2}\\
	&\leq h_{n-2}\big[2A+2^{r/2}B\big(2|w|^r+\mu_r\big)\big]\label{DGVineq3},
\end{align}
where in the final step we used inequality (\ref{Tr1}). Due to the equality between (\ref{DGVsol1}) and (\ref{DGVsol2}), the same argument can be used to verify that inequality (\ref{DGVineq3}) is also valid for $w<0$. Applying inequalities (\ref{gn1}), (\ref{gn2}) and (\ref{DGVineq3}) to inequality (\ref{DGVineq}) yields the bound (\ref{p2.24}) for $0\leq r\leq 1$.

Suppose now that $r>1$. When $r>1$, applying inequalities (\ref{Tr1}) and (\ref{Tr2}) to (\ref{DGVineq2}) (and noting that the case $w<0$ is dealt with similarly) yields the bound 
\begin{equation}
	|wf^{(n-1)}(w)|\leq h_{n-2}\big[2A+2^{r/2}B\big(4|w|^r+\mu_r\big)\big],\label{DGVineq4}
\end{equation}
for $|w|>r-1$. We now bound $|wf^{(n-1)}(w)|$ for $|w|\leq r-1$. We begin by noting that, for $n\geq3$, $u(n)=\Gamma((n-1)/2)/\Gamma(n/2)\leq 2/\sqrt{\pi}$, which follows because $u$ is a decreasing function of $n$ on $(3,\infty)$ (see \cite{ismail}) and $u(3)=2/\sqrt{\pi}$. Using this inequality, the bound (\ref{p2.23}) 
yields that, for $|w|\leq r-1$, 
\begin{align}
	|wf^{(n-1)}(w)|&\leq(r-1)|f^{(n-1)}(w)|\leq(r-1)h_{n-2}\big[A+2^{r/2}B (|w|^r+\mu_{r+1})\big]\label{DGVineq5}.
\end{align}
From inequalities (\ref{DGVineq4}) and (\ref{DGVineq5}) we deduce that, for $r>1$ and $w\in\mathbb{R}$,
 \begin{align}\label{flast00}|wf^{(n-1)}(w)|\leq h_{n-2}\big[(r+1)A+2^{r/2}B((r+3)|w|^r+r\mu_{r+1})\big],
\end{align} 
where we used that $\mu_r\leq\mu_{r+1}$ for $r>1$. Finally, we substitute the bounds (\ref{gn1}), (\ref{gn2}) and (\ref{flast00}) into inequality (\ref{DGVineq}}) to obtain inequality (\ref{p2.24}) for $r>1$,
where we again used that $\mu_r\leq\mu_{r+1}$ for $r>1$. 

\vspace{2mm}

\noindent{(iv)} In obtaining inequality (\ref{aihp1}), \cite{gaunt normal} used the assumption that $g\in C_P^n(\mathbb{R}^d)$ to bound the absolute value of the $n$-th order partial derivatives of $(h\circ g)(\mathbf{w})$ by $h_nP(\mathbf{w})$, $\mathbf{w}\in\mathbb{R}^d$ (see Lemma 2.1 of \cite{gaunt normal}). However, if $h(w)=w$, then under the assumption $g\in C_{P,*}^n(\mathbb{R}^d)$ we can bound the $n$-th order partial derivatives of $(h\circ g)(\mathbf{w})=g(\mathbf{w})$ by $P(\mathbf{w})$, $\mathbf{w}\in\mathbb{R}^d$. Also, note that if $h(w)=w$, then $h_n=1$. We therefore see that inequality (\ref{p2.21}) holds under the weaker assumption $g\in C_{P,*}^n(\mathbb{R}^d)$ if $h(w)=w$. For very similar reasons, we can also weaken the conditions on $g$ in parts (ii) and (iii) of the proposition if $h(w)=w$.

To prove the optimality of the polynomial growth rate in the bounds (\ref{p2.21})--(\ref{p2.24}) it suffices to consider the case $d=1$. Let $h(w)=w$ and $g(w)=|w|^q$, where $q\geq n\geq1$. Then $g\in C_{P_1,*}^n(\mathbb{R})$, where $P_1(w)=n!|w|^{q-n}$. Using the representation (\ref{mvnsolnh}) of $f$ and the dominated convergence theorem, we have that
\begin{align*}f^{(n)}(w)&=-\int_0^1\int_{-\infty}^\infty t^{n-1}g^{(n)}\big(tw+\sqrt{1-t^2}y\big)\phi(y)\,\mathrm{d}y\,\mathrm{d}t\\
&=-n!\int_0^1\int_{-\infty}^\infty t^{n-1}\big|tw+\sqrt{1-t^2}y\big|^{q-n}\big(\mathrm{sgn}\big(tw+\sqrt{1-t^2}y\big)\big)^n\phi(y)\,\mathrm{d}y\,\mathrm{d}t,
\end{align*}
where $\phi$ is the standard normal probability density function and $\mathrm{sgn}(x)$ is the sign of $x\in\mathbb{R}$. A simple asymptotic analysis now gives that
$|f^{(n)}(w)|\sim(n!/q)|w|^{q-n},
$ as $|w|\rightarrow\infty$. Thus, the $|w|^r$ growth rate in inequality (\ref{p2.21}) is optimal. The optimality of the growth rate in inequalities (\ref{p2.22})--(\ref{p2.24}) is established similarly. We observe that $g\in C_{P_2,*}^{n-1}(\mathbb{R})$ and $g\in C_{P_3,*}^{n-2}(\mathbb{R})$, where $P_2(w)=(n-1)!|w|^{q-n+1}$ and $P_3(w)=(n-2)!|w|^{q-n+2}$, and almost identical calculations now confirm that the growth rate in inequalities (\ref{p2.22})--(\ref{p2.24})  is optimal. \qed

\vspace{3mm}

\noindent{\emph{Proof of Proposition \ref{psiprop}.}} (i) Suppose that $\Sigma$ is non-negative definite. We will make use of the following bound given in Lemma 2.4 of \cite{gaunt normal}, which states (under a change of variable in the integral) that, for non-negative definite $\Sigma$, and $g\in C_P^{m+n}(\mathbb{R}^d)$ and $h\in C_b^{m+n}(\mathbb{R})$,
\begin{align}
	\bigg|\frac{\partial^n\psi_m(\mathbf{w})}{\prod_{j=1}^{n}\partial w_{i_j}}\bigg|\leq h_{m+n}\int_{0}^{1}\int_{0}^{1}t^{m+n-1}s^{n-1}\mathbb{E}[P(\mathbf{z}_{s,t,\mathbf{w}}^{\Sigma^{1/2}\mathbf{Z},\Sigma^{1/2}\mathbf{Z}'})]\,\mathrm{d}s\,\mathrm{d}t,\label{aihp3}
\end{align}
where $\mathbf{z}_{s,t,\mathbf{w}}^{\mathbf{x},\mathbf{y}}=st\mathbf{w}+t\sqrt{1-s^2}\mathbf{y}+\sqrt{1-t^2}\mathbf{x}$ and, here and in part (ii) of the proof, $\mathbf{Z}'$ is an independent copy of $\mathbf{Z}$. We will apply inequality (\ref{aihp3}) with dominating function $P(\mathbf{w})=A+B\sum_{i=1}^{d}|w_i|^{r_i}$. Now, using inequality (\ref{abcineq}) gives that
\begin{align}\label{zbound}|(\mathbf{z}_{s,t,\mathbf{w}}^{\mathbf{x},\mathbf{y}})_i|^{r_i}&\leq (st +t\sqrt{1-s^2}+\sqrt{1-t^2})^{r_i}(|w_i|^{r_i}+|x_i|^{r_i}+|y_i|^{r_i})\nonumber\\
&\leq 3^{r_i/2}(|w_i|^{r_i}+|x_i|^{r_i}+|y_i|^{r_i}),
\end{align}
where we used basic calculus to bound $st +t\sqrt{1-s^2}+\sqrt{1-t^2}\leq\sqrt{3}$ for $0<s,t<1$. We therefore obtain the bound
\begin{align*}
	\bigg|\frac{\partial^n\psi_m(\mathbf{w})}{\prod_{j=1}^{n}\partial w_{i_j}}\bigg|&\leq h_{m+n}\int_{0}^{1}\int_{0}^{1}t^{m+n-1}s^{n-1}\bigg[A+B\sum_{i=1}^{d}3^{r_i/2}\big(|w_i|^{r_i}+\mathbb{E}[|(\Sigma^{1/2}\mathbf{Z})_i|^{r_i}]\\
	&\quad+\mathbb{E}[|(\Sigma^{1/2}\mathbf{Z}')_i|^{r_i}]\big)\bigg]\,\mathrm{d}s\,\mathrm{d}t,
\end{align*}
and on evaluating  the integral
we deduce inequality (\ref{p2.31}). 

\vspace{2mm}

\noindent{(ii)} Suppose now that $\Sigma$ is positive definite. Under this assumption, we recall a bound from Lemma 2.4 of \cite{gaunt normal} that states that, for $g\in C_P^{m+n-2}(\mathbb{R}^d)$ and $h\in C_b^{m+n-2}(\mathbb{R})$,
\begin{align}
	\bigg|\frac{\partial^n\psi_m(\mathbf{w})}{\prod_{j=1}^{n}\partial w_{i_j}}\bigg|&\leq h_{m+n-2}\min_{1\leq k,l\leq d}\int_{0}^{1}\int_{0}^{1}\frac{t^{m+n-1}}{\sqrt{1-t^2}}\frac{s^{n-1}}{\sqrt{1-s^2}}\nonumber\\
	&\quad\times\mathbb{E}\big|(\Sigma^{-1/2}\mathbf{Z})_k(\Sigma^{-1/2}\mathbf{Z}')_lP(\mathbf{z}_{s,t,\mathbf{w}}^{\Sigma^{1/2}\mathbf{Z},\Sigma^{1/2}\mathbf{Z}'})\big|\,\mathrm{d}s\,\mathrm{d}t.\label{aihp4}
\end{align}
Applying inequality (\ref{zbound}) to the bound (\ref{aihp4}) gives that 
\begin{align*}
	\bigg|\frac{\partial^n\psi_m(\mathbf{w})}{\prod_{j=1}^{n}\partial w_{i_j}}\bigg|&\leq h_{m+n-2}\min_{1\leq k,l\leq d}\int_{0}^{1}\int_{0}^{1}\frac{t^{m+n-1}}{\sqrt{1-t^2}}\frac{s^{n-1}}{\sqrt{1-s^2}}\mathbb{E}\bigg[|(\Sigma^{-1/2}\mathbf{Z})_k(\Sigma^{-1/2}\mathbf{Z}')_l|\\
	&\quad\times\bigg(A+B\sum_{i=1}^{d}3^{r_i/2}\big(|w_i|^{r_i}+|(\Sigma^{1/2}\mathbf{Z})_i|^{r_i}+|(\Sigma^{1/2}\mathbf{Z}')_i|^{r_i}\big)\bigg)\bigg]\,\mathrm{d}s\,\mathrm{d}t,
\end{align*}
from which inequality (\ref{p2.32}) follows on evaluating the integral similarly to how we did in proving inequality (\ref{p2.22}) of Proposition \ref{fprop}.  The simplified bound inequality (\ref{p2.33}) now follows by setting $\Sigma=I_d$ in inequality (\ref{p2.32}) and using that, for such $\Sigma$ we have $\mathbb{E}|(\Sigma^{-1/2}\mathbf{Z})_k|=\mathbb{E}|Z|=\sqrt{2/\pi}$ for $k=1,\ldots,d$.

\vspace{2mm}

\noindent{(iii)} We proceed similarly to the proof of inequality (\ref{p2.24}), using the iterative technique of \cite{dgv17}. Recall that $L\psi_m(w)=f^{(m)}(w)$, where $L$ is the standard normal Stein operator. Differentiating 
gives that 
\begin{equation}
	L\psi_m'(w)=f^{(m+1)}(w)+\psi_m'(w).\label{dvg}
\end{equation}
From (\ref{dvg}) and the triangle inequality, we get that 
\begin{equation}
	|\psi_m^{(3)}(w)|\leq|w\psi_m''(w)|+|f^{(m+1)}(w)|+|\psi_m'(w)|\label{dgvineq1}.
\end{equation}
From (\ref{p2.24}) and (\ref{p2.33}), we have the bounds 
\begin{align}|f^{(m+1)}(w)|&\leq 2h_{m-1}\big[2A+2^{r/2}B(2|w|^r+\mu_{r})\big], \quad 0\leq r\leq1,\label{dgvineq2aaa} \\
	|f^{(m+1)}(w)|&\leq h_{m-1}\big[(r+3)A+2^{r/2}B((r+5)|w|^r+(r+1)\mu_{r+1})\big],\quad r>1,\label{dgvineq2} \\
	|\psi_m'(w)|
	&\leq h_{m-1}\big[A+3^{r/2}B(|w|^r+2\mu_{r+1})\big], \quad r\geq0,\label{dgvineq3}
\end{align}
where we used that, for $m\geq2$, $v(m)=\Gamma((m+1)/2)/\Gamma(m/2+1)\leq\sqrt{\pi}/2$, which follows because $v$ is a decreasing function of $m$ on $(2,\infty)$ (see \cite{ismail}) and $v(2)=\sqrt{\pi}/2$, in applying inequality (\ref{p2.33}). 

It now remains to bound $|w\psi_m''(w)|$, which requires more effort to establish a bound with the correct $w$-dependence.
From (\ref{n01}), (\ref{n02}) and (\ref{dvg}) we obtain the following representation for $\psi_m''(w)$:
\begin{align}
	\psi_m''(w)&=-\mathrm{e}^{w^2/2}\int_{w}^{\infty}\big[f^{(m+1)}(t)+\psi_m'(t)\big]\mathrm{e}^{-t^2/2}\,\mathrm{d}t\label{dgv1}\\
	&=\mathrm{e}^{w^2/2}\int_{-\infty}^{w}\big[f^{(m+1)}(t)+\psi_m'(t)\big]\mathrm{e}^{-t^2/2}\,\mathrm{d}t\label{dgv2}.
	\end{align}
 We first consider the case $0\leq r\leq1$. 
 Applying (\ref{dgvineq2aaa}) and (\ref{dgvineq3}) to (\ref{dgv1}) gives that, for $w\geq0$, 
\begin{align}|w\psi_m''(w)|&\leq w\mathrm{e}^{w^2/2}\int_{w}^{\infty}\big[|f^{(m+1)}(t)|+|\psi_m'(t)|\big]\mathrm{e}^{-t^2/2}\,\mathrm{d}t\label{wpsi2}\\
	&\leq h_{m-1}w\mathrm{e}^{w^2/2}\int_{w}^{\infty}\big[5A+3^{r/2}B(5t^r+2\mu_{r+1}+2\mu_r)\big]\mathrm{e}^{-t^2/2}\,\mathrm{d}t\nonumber\\
\label{req1}	&\leq h_{m-1}\big[5A+3^{r/2}B(5|w|^r+2\mu_{r+1}+2\mu_r)\big],
\end{align}
where we used inequality (\ref{Tr1}) to bound the integral. Due to the equality between (\ref{dgv1}) and (\ref{dgv2}), it is readily seen that inequality (\ref{req1}) also holds for $w<0$.  Applying inequalities (\ref{dgvineq2aaa}), (\ref{dgvineq3}) and (\ref{req1}) to inequality (\ref{dgvineq1}), as well as using the inequality $2\mu_r\leq 3\mu_{r+1}$, $0\leq r\leq1$, to simplify the bound, now yields (\ref{p2.34}) for $0\leq r\leq1$. 
 
Suppose now that $r>1$. We now apply inequalities (\ref{dgvineq2}) and (\ref{dgvineq3}) to (\ref{wpsi2}) to get that, for $|w|>r-1$,
\begin{align}
	|w\psi_m''(w)|
	&\leq h_{m-1}|w|\mathrm{e}^{w^2/2}\int_{|w|}^{\infty}\big[(r+4)A+3^{r/2}B((r+6)t^r+(r+3)\mu_{r+1})\big]\mathrm{e}^{-t^2/2}\,\mathrm{d}t\nonumber
	\\
	&\leq h_{m-1}\big[(r+4)A+3^{r/2}B((2r+12)|w|^r+(r+3)\mu_{r+1})\big],\label{dgvineq5}
	\end{align}
where we used inequalities (\ref{Tr1}) and (\ref{Tr2}) to bound the integral. Now suppose that $|w|\leq r-1$. Rearranging $L\psi_m(w)=f^{(m)}(w)$ and applying the triangle inequality gives that, for $|w|\leq r-1$, 
\begin{align*}
	|w\psi_m''(w)|\leq|w^2\psi_m'(w)|+|wf^{(m)}(w)|\leq (r-1)^2|\psi_m'(w)|+(r-1)|f^{(m)}(w)|.
	\end{align*}
Using inequalities (\ref{p2.23}) and (\ref{p2.33}) now gives that
\begin{align}
	|w\psi_m''(w)|&\leq (r-1)^2h_{m-1}\big[A+3^{r/2}B(|w|^r+2\mu_{r+1})\big]\nonumber\\
	&\quad+(r-1)h_{m-1}\big[A+2^{r/2}B(|w|^r+\mu_{r+1})\big] \nonumber\\
\label{wppppp}	&\leq h_{m-1}\big[r^2A+3^{r/2}B(r^2|w|^r+2r^2\mu_{r+1})\big],
\end{align}
where in applying inequalities (\ref{p2.23}) and (\ref{p2.33}) we used the bounds $\Gamma((m+1)/2)/\Gamma(m/2+1)\leq\sqrt{\pi}/2$ and $\Gamma(m/2)/\Gamma((m+1)/2)\leq2/\sqrt{\pi}$ for $m\leq2$, which can be verified with the usual argument for bounding ratios of gamma functions that we have used earlier in this paper. From inequalities  (\ref{dgvineq5}) and (\ref{wppppp}) we deduce that, for $r>1$ and $w\in\mathbb{R}$,
\begin{align}\label{wpppppz}	|w\psi_m''(w)|\leq h_{m-1}\big[(r^2+4)A+3^{r/2}B((r^2+r+12)|w|^r+(2r^2+2)\mu_{r+1})\big].
\end{align}
 Finally, we substitute the bounds (\ref{dgvineq2}), (\ref{dgvineq3}) and (\ref{wpppppz}) into inequality (\ref{dgvineq1}) to obtain inequality (\ref{p2.34}) for $r>1$.

\vspace{2mm}

\noindent{(iv)} When $h(w)=w$, the bounds in parts (i), (ii) and (iii) hold for $g$ in the weaker classes $C_{P,*}^{m+n}(\mathbb{R}^d)$, $C_{P,*}^{m+n-2}(\mathbb{R}^d)$ and $C_{P,*}^{m-1}(\mathbb{R})$ for the same reason given in part (iv) of the proof of Proposition \ref{fprop}. 

The optimality of the $|w|^r$ growth rate of inequalities (\ref{p2.31})--(\ref{p2.34}) is established similarly to we did in part (iv) of the proof of Proposition \ref{fprop}. We demonstrate the optimality of the growth rate in inequality (\ref{p2.31}); similar considerations show the optimality of the growth rate in inequalities (\ref{p2.32})--(\ref{p2.34}). Let $d=1$, $h(w)=w$ and $g(w)=|w|^q$, where $q\geq m+n$. Observe that $g\in C_{P,*}^{m+n}(\mathbb{R})$ with $P(w)=(m+n)!|w|^{q-m-n}$. Also, by the dominated convergence theorem,
\begin{align*}\psi_m^{(n)}(w)&=\int_0^1\!\int_0^1\!\int_{-\infty}^\infty\!\int_{-\infty}^\infty t^{m+n-1}s^{n-1}g^{(m+n)}(z_{s,t,w}^{x,y})\phi(x)\phi(y)\,\mathrm{d}x\,\mathrm{d}y\,\mathrm{d}s\,\mathrm{d}t\\
&=(m+n)!\int_0^1\!\int_0^1\!\int_{-\infty}^\infty\!\int_{-\infty}^\infty t^{m+n-1}s^{n-1}|z_{s,t,w}^{x,y}|^{q-m-n}(\mathrm{sgn}(z_{s,t,w}^{x,y}))^{m+n}\times\\
&\quad\times\phi(x)\phi(y)\,\mathrm{d}x\,\mathrm{d}y\,\mathrm{d}s\,\mathrm{d}t,
\end{align*}
where we recall that $z_{s,t,w}^{x,y}=stw+t\sqrt{1-s^2}y+\sqrt{1-t^2}x$. A simple asymptotic analysis now gives that
$|\psi_m^{(n)}(w)|\sim(m+n)!|w|^{q-m-n}/(q(q-m)),
$ as $|w|\rightarrow\infty$. Thus, the $|w|^r$ growth rate in inequality (\ref{p2.31}) is optimal.
\qed

\section{Bounds for the distance between the distributions of $g(\mathbf{W})$ and $g(\mathbf{Z})$}\label{sec3}

In this section, we obtain general bounds to quantify the quality of the distributional approximation of $g(\mathbf{W})$ by $g(\mathbf{Z})$, where $\mathbf{Z}\sim\mathrm{MVN}_d(\mathbf{0},I_d)$, in the setting that $\mathbf{W}$ is a sum of independent random vectors with independent components. We shall suppose that $X_{1,1},\ldots,X_{n_1,1},\ldots,X_{1,d},\ldots,X_{n_d,d}$ are independent random variables, and define  the random vector $\mathbf{W}:=(W_1,\ldots,W_d)^\intercal$, where $W_j=n_j^{-1/2}\sum_{i=1}^{n_j}X_{ij}$, for $1\leq j\leq d$. We shall also assume that $\mathbb{E}[X_{ij}^k]=\mathbb{E}[Z^k]$ for all $1\leq i\leq n_j$, $1\leq j\leq d$ and all $k\in\mathbb{Z}^+$ such that $k\leq p$, for some $p\geq2$; having three or more matching moments allows for faster convergence rates than the $O(n^{-1/2})$ Berry-Esseen rate. As in Section \ref{sec2}, we shall suppose that the partial derivatives of $g$ up to a specified order exist and have polynomial growth rate. To this end, we introduce the dominating function $P(\mathbf{w})=A+B\sum_{i=1}^{d}|w_i|^{r_i},$ where $A, B$ and $r_1,\ldots, r_d$ are non-negative constants. 
In the univariate $d=1$ case, we simplify notation, writing $W=n^{-1/2}\sum_{i=1}^nX_i$, where $X_1,\ldots,X_n$ are independent random variables such that $\mathbb{E}[X_i^k]=\mathbb{E}[Z^k]$ for all $1\leq i\leq n$ and $1\leq k\leq p$. The dominating function also takes the simpler form, $P(w)=A+B|w|^r$.

Our general bounds are stated in the following theorem, and improve on the bounds of Theorems 3.2--3.5 of \cite{gaunt normal} through smaller numerical constants and weaker moment conditions. These improvements are a result of our improved bounds of Propositions \ref{fprop} and \ref{psiprop} on the solutions of the Stein equations (\ref{mvng}) and (\ref{mvnpsi}) (the improved $w$-dependence of the univariate bounds results in weaker moment conditions) and some more careful simplifications in the proof of Theorem \ref{thm3.2} to improve the dependence of the bounds in parts (iii) and (iv) (in which $g$ is an even function) on the moments of the $X_{ij}$. The rate of convergence of all bounds with respect to $n$ is of optimal order; see \cite[Proposition 3.1]{gaunt normal}. We shall let $\Delta_h(g(\mathbf{W}),g(\mathbf{Z}))$ denote the quantity $|\mathbb{E}[h(g(\mathbf{W}))]-\mathbb{E}[h(g(\mathbf{Z}))]|$.  The bounds involve the constants $c_r=\max\{1,2^{r-1}\},$ $r\geq 0$.


\begin{theorem}\label{thm3.2} Suppose that the above notations and assumptions prevail. Then under additional assumptions, as given below, the following bounds hold. 

\vspace{2mm}

\noindent{(i)} Suppose $\mathbb{E}|X_{ij}|^{r_l+p+1}<\infty$ for all $i,j,l$, and that $g\in C_P^p(\mathbb{R}^d)$ and $h\in C_b^p(\mathbb{R})$. Then
	\begin{align}
		&\Delta_h(g(\mathbf{W}),g(\mathbf{Z}))\nonumber\\
		&\leq \frac{(p+1)\sqrt{\pi}\Gamma(\frac{p+1}{2})}{2p!\Gamma(\frac{p}{2}+1)}h_p\sum_{j=1}^{d}\sum_{i=1}^{n_j}\frac{1}{n_j^{(p+1)/2}}\bigg[A\mathbb{E}|X_{ij}|^{p+1}+B\sum_{k=1}^{d}2^{r_k/2}\bigg(c_{r_k}\mathbb{E}|X_{ij}|^{p+1}\mathbb{E}|W_k|^{r_k}\nonumber\\
	\label{thm3.1b}	&\quad+\frac{c_{r_k}}{n_k^{r_k/2}}\mathbb{E}|X_{ij}^{p+1}X_{ik}^{r_k}|+\mu_{r_k+1}\mathbb{E}|X_{ij}|^{p+1}\bigg)\bigg].
	\end{align}
\noindent{(ii)} Suppose $\mathbb{E}|X_{i}|^{r+p+1}<\infty$ for all $i$, and that $g\in C_P^{p-1}(\mathbb{R})$ and $h\in C_b^{p-1}(\mathbb{R})$. Then  
		\begin{align}
		\Delta_h(g(W),g(Z))
		&\leq \frac{(p+1)}{p!n^{(p+1)/2}}h_{p-1}\sum_{i=1}^{n}\bigg[\alpha_r A\mathbb{E}|X_i|^{p+1}+2^{r/2}B\bigg(c_r\beta_r\bigg(\mathbb{E}|X_i|^{p+1}\mathbb{E}|W|^r\nonumber\\
\label{thm3.3b}		&\quad+\frac{1}{n^{r/2}}\mathbb{E}|X_i|^{r+p+1}\bigg)+\gamma_r\mathbb{E}|X_i|^{p+1}\bigg)\bigg].
	\end{align}
\noindent{(iii)} Suppose $\mathbb{E}|X_{ij}|^{r_l+p+2}<\infty$ for all $i,j,l$, and that $g\in C_P^{p+2}(\mathbb{R}^d)$ and $h\in C_b^{p+2}(\mathbb{R})$. In addition, suppose that $p\geq2$ is even and that $g$ is an even function. 
Then	
	\begin{align}
		&\Delta_h(g(\mathbf{W}),g(\mathbf{Z}))\nonumber\\
		&\leq\frac{1}{p!}h_{p+2}\bigg\{\sum_{j=1}^{d}\sum_{i=1}^{n_j}\frac{1}{n_j^{p/2+1}}\frac{2p+3}{(p+1)(p+2)}\bigg[A\mathbb{E}|X_{ij}|^{p+2}+B\sum_{k=1}^{d}2^{r_k/2}\bigg(c_{r_k}\mathbb{E}|X_{ij}|^{p+2}\mathbb{E}|W_k|^{r_k}\nonumber\\
		&\quad+\frac{c_{r_k}}{n_k^{r_k/2}}\mathbb{E}|X_{ij}^{p+2}X_{ik}^{r_k}|+\mu_{r_k}\mathbb{E}|X_{ij}|^{p+2}\bigg)\bigg]+\frac{3\pi\Gamma(\frac{p}{2}+2)}{8\sqrt{2}\Gamma(\frac{p+5}{2})}\sum_{j=1}^{d}\sum_{i=1}^{n_j}\frac{|\mathbb{E}[X_{ij}^{p+1}]|}{n_j^{(p+1)/2}}\nonumber\\
		&\quad\times\sum_{k=1}^{d}\sum_{l=1}^{n_k}\frac{1}{n_k^{3/2}}\bigg[A\mathbb{E}|X_{lk}|^3+B\sum_{t=1}^{d}3^{r_t/2}\bigg(c_{r_t}\mathbb{E}|X_{lk}|^3\mathbb{E}|W_t|^{r_t}+\frac{c_{r_t}}{n_t^{r_t/2}}\mathbb{E}|X_{lk}^3X_{lt}^{r_t}|\nonumber\\
\label{thm3.4b}		&\quad+2\mu_{r_t+1}\mathbb{E}|X_{lk}|^3\bigg)\bigg]\bigg\}.
	\end{align}
\noindent{(iv)} Suppose $\mathbb{E}|X_{i}|^{r+p+2}<\infty$ for all $i$, and that $g\in C_P^{p}(\mathbb{R})$ and $h\in C_b^{p}(\mathbb{R})$. In addition, suppose that $p\geq2$ is even and that $g$ is an even function. Then 	
		\begin{align}
		\Delta_h(g(W),g(Z))
		&\leq\frac{1}{p!n^{p/2+1}}h_p\bigg\{\sum_{i=1}^{n}\frac{2p+3}{p+1}\bigg[\alpha_r A\mathbb{E}|X_i|^{p+2}+2^{r/2}B\bigg(c_r\beta_r\bigg(\mathbb{E}|X_i|^{p+2}\mathbb{E}|W|^r\nonumber\\
		&\quad+\frac{1}{n^{r/2}}\mathbb{E}|X_i|^{r+p+2}\bigg)\!+\!\gamma_r\mathbb{E}|X_i|^{p+2}\bigg)\bigg]\!+\!\frac{3}{2n}\sum_{i=1}^{n}\sum_{l=1}^{n}|\mathbb{E}[X_i^{p+1}]|\bigg[\tilde\alpha_r A\mathbb{E}|X_l|^3\nonumber\\
\label{thm3.5b}		&\quad+3^{r/2}B\bigg(c_r\tilde\beta_r\bigg(\mathbb{E}|X_l|^3\mathbb{E}|W|^r
			+\frac{1}{n^{r/2}}\mathbb{E}|X_l|^{r+3}\bigg)\!+\tilde\gamma_r\mathbb{E}|X_l|^3\bigg)\bigg]\bigg\}.
	\end{align}
	\noindent{(v)} If $h(w)=w$ for all $w\in\mathbb{R}$, then the inequalities in parts (i)--(iv) hold for $g$ in the classes $C_{P,*}^p(\mathbb{R}^d)$, $C_{P,*}^{p-1}(\mathbb{R})$, $C_{P,*}^{p+2}(\mathbb{R}^d)$ and $C_{P,*}^{p}(\mathbb{R})$, respectively.
\end{theorem}

If one only requires bounds with an explicit dependence on $n_1,\ldots,n_d$ and the dimension $d$, 
but not with explicit constants, then the bounds given in the following corollary may be preferable.

\begin{corollary}\label{simplebds}	Let $r_*=\max_{1\leq j\leq d}r_j$ and $n_*=\min_{1\leq j\leq d}n_j$, and also let $\tilde{h}_p=\sum_{j=1}^p\|h^{(j)}\|$. Let $C$ be a constant that does not depend on $n_1,\ldots,n_d$ and $d$ and which may change from line to line. Then
	
	\vspace{2mm}
	
		\noindent{(i)} Under the assumptions of part (i) of Theorem \ref{thm3.2}, 
		\begin{equation}
	\label{dd1}		\Delta_h(g(\mathbf{W}),g(\mathbf{Z}))\leq \frac{Cd\tilde{h}_p}{n_*^{(p+1)/2}}\sum_{j=1}^d\sum_{i=1}^{n_j}\mathbb{E}|X_{ij}|^{r_*+p+1}.
		\end{equation}
		(ii) Under the assumptions of part (ii) of Theorem \ref{thm3.2},
		\begin{equation*}
			\Delta_h(g(W),g(Z))\leq \frac{C\tilde{h}_{p-1}}{n^{(p+1)/2}}\sum_{i=1}^n\mathbb{E}|X_{i}|^{r+p+1}.
		\end{equation*}
		(iii) Under the assumptions of part (iii) of Theorem \ref{thm3.2}, 
		\begin{align}
		\label{dd2}	\Delta_h(g(\mathbf{W}),g(\mathbf{Z}))&\leq \frac{Cd\tilde{h}_{p+2}}{n_*^{p/2+2}}\sum_{j=1}^d\sum_{k=1}^d\sum_{i=1}^{n_j}\sum_{l=1}^{n_k}\big(\mathbb{E}|X_{ij}|^{r_*+p+2}+|\mathbb{E}[X_{ij}^{p+1}]|\mathbb{E}|X_{lk}|^{r_*+3} \big).
		\end{align}
		(iv) Under the assumptions of part (iv) of Theorem \ref{thm3.2},
		\begin{equation*}
			\Delta_h(g(W),g(Z))\leq \frac{C\tilde{h}_p}{n^{p/2+2}}\sum_{i=1}^n\sum_{l=1}^n\big(\mathbb{E}|X_i|^{r+p+2}+|\mathbb{E}[X_i^{p+1}]|\mathbb{E}|X_l|^{r+3} \big).
		\end{equation*}
	\end{corollary}
	

\begin{remark}The univariate bounds in parts (ii) and (iv) give bounds on $\Delta_h(g(W),g(Z))$ that hold under weaker assumptions on $g$ and $h$ than the bounds in parts (i) and (iii) with $d=1$. The univariate bounds of Theorems 3.2--3.5 of \cite{gaunt normal} also enjoyed this property, but came at the cost of stronger moment conditions in the univariate case. In virtue of the improved $w$-dependence of our univariate bounds in Propositions \ref{fprop} and \ref{psiprop}, we were able to derive univariate bounds that do not incur this cost.
\end{remark}

\begin{remark}\label{vdfeb}
As noted by \cite[Remark 3.6]{gaunt normal}, if $\|h^{(k)}\|=1$, $1\leq k\leq p$, then $h_p=\sum_{k=1}^p{p\brace k}=B_p$, where $B_p=\mathrm{e}^{-1}\sum_{j=1}^\infty j^p/j!$ is the $p$-th Bell number (see \cite[Section 26.7(i)]{olver}). For example, setting $h_p=B_p$ in inequality (\ref{thm3.1b}) gives a bound in the $d_{p}$ metric, $p\geq2$. Also note that setting $p=2$ in part (ii) gives a bound in the Wasserstein distance. To understand the behaviour of the bound (\ref{thm3.1b}) for large $p$, we apply the lower
bound in (\ref{gamma1}), Stirling's inequality $p!>\sqrt{2\pi}p^{p+1/2}\mathrm{e}^{-p}$ and the inequality $B_p<(0.792p/\log(p+1))^p$ of \cite{bt10} to obtain the bound
\begin{align*}\frac{\sqrt{\pi}\Gamma(\frac{p+3}{2})}{p!\Gamma(\frac{p}{2}+1)}h_p<\frac{1}{2}\sqrt{\frac{p+2}{p}}\bigg(\frac{2.153}{\log(p+1)}\bigg)^p\leq\frac{1}{\sqrt{2}}\bigg(\frac{2.153}{\log(p+1)}\bigg)^p, \quad p\geq2.
\end{align*}
We note that whilst we have shown that this constant tends to zero as $p\rightarrow\infty$, this does not imply that the bound (\ref{thm3.1b}) will tend to zero in this limit if $n$ is fixed. This is because, for $p\geq2$ and $i=1,\ldots,n$, we have that $\mathbb{E}|X_i|^{p+1}\geq\mathbb{E}|X_i|^p=\mathbb{E}|Z|^p=2^{p/2}\Gamma((p+1)/2)/\sqrt{\pi}$. Similar comments apply to parts (ii)--(iv) of Theorem \ref{thm3.2}.
\end{remark}

\begin{remark}Suppose $n_1=\cdots=n_d=n$. For a fixed number of matching moments $p\geq2$, if we allow the dimension $d$ to grow with $n$, the bound (\ref{dd1}) of Corollary \ref{simplebds} tends to zero if $n/d^{4/(p-1)}\rightarrow\infty$ and the bound (\ref{dd2}) tends to zero if $n/d^{6/p}\rightarrow\infty$. In particular, for $p\geq6$, the bound (\ref{dd1}) can tend to zero even if $d\gg n$.
\end{remark}

We now set about proving Theorem \ref{thm3.2} and Corollary \ref{simplebds}. We begin with the following lemma, which provides bounds on some expectations that will appear in the proof of Theorem \ref{thm3.2}. 
In the lemma, $f$ denotes the solution (\ref{mvnsolnh}) of the Stein equation (\ref{mvng}) and $\psi_{m,j}$, for $m\geq1$, $1\leq j\leq d$, is the solution of the Stein equation 
\begin{equation}\label{mvnpsi2}
	\nabla^\intercal\Sigma\nabla \psi_{m,j}(\mathbf{w})-\mathbf{w}^\intercal\nabla \psi_{m,j}(\mathbf{w})=\frac{\partial^m f(\mathbf{w})}{\partial w_j^m}.
\end{equation}
In the univariate $d=1$ case, we drop the subscript $j$, and simply write $\psi_m$ for the solution to (\ref{mvnpsi2}). For $1\leq i\leq n$ and $1\leq j\leq d$, we let $\mathbf{W}^{(i,j)}=\mathbf{W}-n_j^{-1/2}\mathbf{X}_{ij}$, where the random vector $\mathbf{X}_{ij}$ is defined to be such that it has $X_{ij}$ as its $j$-th entry and the other $d-1$ entries are equal to zero. Note that $\mathbf{W}^{(i,j)}$ is independent of $\mathbf{X}_{ij}$. We then define $\mathbf{W}_\theta^{(i,j)}=\mathbf{W}^{(i,j)}+\theta n_j^{-1/2}\mathbf{X}_{ij}$ for some $\theta\in (0,1)$. In the univariate case, we let $W_\theta^{(i)}=W^{(i)}+\theta X_{i}/\sqrt{n}$, where $W^{(i)}=W-X_i/\sqrt{n}$.

\begin{lemma} \label{lem3.6} Let $P(\mathbf{w})=A+B\sum_{i=1}^{d}|w_i|^{r_i},$ where $A,B\geq0$, and $r_1,\ldots, r_d\geq0$. Assume that $\Sigma=I_d$, $\theta\in(0,1)$ and $q\geq0$, $m\geq2$ and $t\geq3$.  Then 
	\begin{align}\mathbb{E}\bigg|X_{ij}^q\frac{\partial^tf}{\partial w_j^t}(\mathbf{W}_\theta^{(i,j)})\bigg| &\leq \frac{h_t}{t}\bigg[A\mathbb{E}|X_{ij}|^q+B\sum_{k=1}^d2^{r_k/2}\bigg(c_{r_k}\mathbb{E}|X_{ij}|^q\mathbb{E}|W_k|^{r_k}\nonumber\\
		&\quad+\frac{c_{r_k}}{n_k^{r_k/2}}\mathbb{E}|X_{ij}^qX_{ik}^{r_k}|+\mu_{r_k}\mathbb{E}|X_{ij}|^q\bigg)\bigg],\label{l3.61}\\
		\mathbb{E}\bigg|X_{ij}^q\frac{\partial^tf}{\partial w_j^t}(\mathbf{W}_\theta^{(i,j)})\bigg| &\leq h_{t-1}\frac{\sqrt{\pi}\Gamma(\frac{t}{2})}{2\Gamma(\frac{t+1}{2})}\bigg[A\mathbb{E}|X_{ij}|^q+B\sum_{k=1}^d2^{r_k/2}\bigg(c_{r_k}\mathbb{E}|X_{ij}|^q\mathbb{E}|W_k|^{r_k}\nonumber\\
		&\quad+\frac{c_{r_k}}{n_k^{r_k/2}}\mathbb{E}|X_{ij}^qX_{ik}^{r_k}|+\mu_{r_k+1}\mathbb{E}|X_{ij}|^q\bigg)\bigg],\label{l3.62}\\
		\mathbb{E}\bigg|X_{ij}^q\frac{\partial^3\psi_{m,j}}{\partial w_j^3}(\mathbf{W}_\theta^{(i,j)})\bigg| &\leq h_{m+1}\frac{\pi\Gamma(\frac{m+3}{2})}{4\sqrt{2}\Gamma(\frac{m}{2}+2)}\bigg[A\mathbb{E}|X_{ij}|^q+B\sum_{k=1}^d3^{r_k/2}\bigg(c_{r_k}\mathbb{E}|X_{ij}|^q\mathbb{E}|W_k|^{r_k}\nonumber\\
		&\quad+\frac{c_{r_k}}{n_k^{r_k/2}}\mathbb{E}|X_{ij}^qX_{ik}^{r_k}|+2\mu_{r_k+1}\mathbb{E}|X_{ij}|^q\big)\bigg], \label{l3.63}
	\end{align}
	where the inequalities are for $g$ in the classes $C_P^t(\mathbb{R}^d)$, $C_P^{t-1}(\mathbb{R}^d)$ and $C_P^{m+1}(\mathbb{R}^d)$, respectively.   Now suppose $d=1$ and $\Sigma=1$.  Then
	\begin{align}\mathbb{E}|X_i^qf^{(t)}(W_{\theta}^{(i)})|&\leq h_{t-2}\bigg\{\alpha_r A\mathbb{E}|X_i|^q+2^{r/2}B\bigg[c_r\beta_r\bigg(\mathbb{E}|X_i|^q\mathbb{E}|W|^r\nonumber\\
			&\quad+\frac{1}{n^{r/2}}\mathbb{E}|X_i|^{q+r}\bigg)
	+\gamma_r\mathbb{E}|X_i|^q\bigg]\bigg\}, \label{l3.64}\\
		\mathbb{E}|X_i^q\psi_m^{(3)}(W_{\theta}^{(i)})|&\leq h_{m-1}\bigg\{\tilde\alpha_r A\mathbb{E}|X_i|^q+3^{r/2}B\bigg[c_r\tilde\beta_r\bigg(\mathbb{E}|X_i|^q\mathbb{E}|W|^r\nonumber\\
		&\quad+\frac{1}{n^{r/2}}\mathbb{E}|X_i|^{q+r}\bigg)+\tilde\gamma_r\mathbb{E}|X_i|^q\bigg]\bigg\}\label{l3.65},
	\end{align}
	where the inequalities are for $g$ in the classes $C_P^{t-2}(\mathbb{R})$ and $C_P^{m-1}(\mathbb{R})$, respectively.  
\end{lemma}

\begin{proof} The proof is very similar to that of Lemma 3.4 of \cite{gaunt normal}. The only difference is that we improve the bounds of \cite{gaunt normal} by applying the improved bounds of Propositions \ref{fprop} and \ref{psiprop} for the solutions $f$ and $\psi_m$ and we also use the the inequality $|a+b|^r\leq c_r(|a|^r+|b|^r)$, $r\geq0$, which is sharper than the inequality $|a+b|^r\leq 2^r(|a|^r+|b|^r)$, $r\geq0$, used by \cite{gaunt normal}. 
\end{proof}

\noindent{\emph{Proof of Theorem \ref{thm3.2}.}} Under the assumptions of part (i) of the theorem, the following bound given in Lemma 3.1 of \cite{gaunt normal} holds:
\begin{align} \Delta_h(g(\mathbf{W}),g(\mathbf{Z}))&\leq\sum_{j=1}^d\sum_{i=1}^{n_j}\frac{1}{(p-1)!n_j^{(p+1)/2}}\bigg\{\mathbb{E}\bigg|X_{ij}^{p-1}\frac{\partial^{p+1}f}{\partial w_j^{p+1}}(\mathbf{W}_{\theta_1}^{(i,j)})\bigg|\nonumber\\
	\label{springz}&\quad+\frac{1}{p}\mathbb{E}\bigg|X_{ij}^{p+1}\frac{\partial^{p+1}f}{\partial w_j^{p+1}}(\mathbf{W}_{\theta_2}^{(i,j)})\bigg|\bigg\},
\end{align} 
for some $\theta_1,\theta_2\in(0,1)$. We now obtain inequality (\ref{thm3.1b}) by using inequality (\ref{l3.62}) to bound the expectations in the bound (\ref{springz}), and then simplifying the resulting bound by using the basic inequalities $\mathbb{E}|X_{ij}|\leq\mathbb{E}|X_{ij}|^a\leq\mathbb{E}|X_{ij}|^b$ for $2\leq a\leq b$. These inequalities follow easily from an application H\"older's inequality and the assumption that $\mathbb{E}[X_{ij}^2]=1$. We also use these inequalities to deduce that $\mathbb{E}|X_{ij}^{p-1}X_{ik}^{r_k}|\leq\mathbb{E}|X_{ij}^{p+1}X_{ik}^{r_k}|$ ($p\geq2$, $r_k\geq0$). This is verified by considering the cases $j=k$ and $j\not=k$ separately, where in the latter case we make use of the independence of $X_{ij}$ and $X_{ik}$.
The proof of inequality (\ref{thm3.3b}) is similar, but we instead use inequality (\ref{l3.64}) to bound the expectations in (\ref{springz}).

Now, under the assumptions of part (iii) of the theorem, the following bound of Lemma 3.3 of \cite{gaunt normal} holds:
\begin{align}\Delta_h(g(\mathbf{W}),g(\mathbf{Z}))&\leq\sum_{j=1}^d\sum_{i=1}^{n_j}\frac{1}{p!n_j^{p/2+1}}\bigg\{\mathbb{E}\bigg|X_{ij}^{p}\frac{\partial^{p+2}f}{\partial w_j^{p+2}}(\mathbf{W}_{\theta_1}^{(i,j)})\bigg|+\!\frac{1}{p+1}\!\mathbb{E}\bigg|X_{ij}^{p+2}\frac{\partial^{p+2}f}{\partial w_j^{p+2}}(\mathbf{W}_{\theta_2}^{(i,j)})\bigg|\nonumber\\
	&\quad+\!|\mathbb{E}X_{ij}^{p+1}|\mathbb{E}\bigg|X_{ij}\frac{\partial^{p+2}f}{\partial w_j^{p+2}}(\mathbf{W}_{\theta_3}^{(i,j)})\bigg|\bigg\}+\sum_{j=1}^d\sum_{i=1}^{n_j}\frac{|\mathbb{E}[X_{ij}^{p+1}]|}{p!n_j^{(p+1)/2}}\sum_{k=1}^d\sum_{l=1}^{n_k}\frac{1}{n_k^{3/2}}\times \nonumber\\
	\label{dig hole}& \quad\times\bigg\{\mathbb{E}\bigg|X_{lk}\frac{\partial^3\psi_{p+1,j}}{\partial w_k^3}(\mathbf{W}_{\theta_4}^{(l,k)})\bigg|+\frac{1}{2}\mathbb{E}\bigg|X_{lk}^3\frac{\partial^3\psi_{p+1,j}}{\partial w_k^3}(\mathbf{W}_{\theta_5}^{(l,k)})\bigg|\bigg\},
\end{align}
for some $\theta_1,\ldots,\theta_5\in(0,1)$. Using inequalities (\ref{l3.61}) and (\ref{l3.63}) to bound the expectations in the bound (\ref{dig hole}) and simplifying the resulting bound using the same considerations given earlier in this proof yields the bound (\ref{thm3.4b}). The proof of inequality (\ref{thm3.5b}) is similar, but we instead use the inequalities (\ref{l3.64}) and (\ref{l3.65}) to bound the expectations in (\ref{dig hole}). Finally, the assertion in part (v) is a consequence of part (iv) of Propositions \ref{fprop} and \ref{psiprop}. \qed

\vspace{3mm}

\noindent{\emph{Proof of Corollary \ref{simplebds}.}} 	
	The simplified bounds are obtained by applying the following considerations to the bounds of Theorem \ref{thm3.2}. Firstly, for a real-valued random variable $Y$ with $\mathbb{E}|Y|^b<\infty$, we have $\mathbb{E}|Y|^a\leq\mathbb{E}|Y|^b,$ for any $2\leq a\leq b$ . 
	 Secondly, $\mathbb{E}|W_k|^r\leq Cn_k^{-1}\sum_{i=1}^{n_k}\mathbb{E}|X_{ik}|^{r}$ for $r\geq2$, and $\mathbb{E}|W_k|^r\leq1$ for $0\leq r<2$. For $r\geq2$, this inequality follows from the Marcinkiewiczi-Zygmund inequality \cite{mz37} $\mathbb{E}|\sum_{i=1}^{n}Y_i|^r\leq C_r\{\sum_{i=1}^{n}\mathbb{E}|Y_i|^r+(\sum_{i=1}^{n}\mathbb{E}[Y_i^2])^{1/2}\}$ and the inequality $\mathbb{E}|Y|^a\leq\mathbb{E}|Y|^b,$ for any $2\leq a\leq b$. For $0\leq r<2$, we just use H\"older's inequality, $\mathbb{E}|W_k|^r\leq(\mathbb{E}W_k^2)^{r/2}=1$.
	Thirdly, by H\"{o}lder's inequality, $\mathbb{E}|X_{ij}^aX_{ik}^b|\leq(\mathbb{E}|X_{ij}^{a+b}|)^{a/(a+b)} (\mathbb{E}|X_{ik}^{a+b}|)^{b/(a+b)} \leq\max\{\mathbb{E}|X_{ij}|^{a+b},\mathbb{E}|X_{ik}|^{a+b}\}$, for $a,b\geq1$. 
\qed

\section{Application to chi-square approximation
}\label{sec4}

In this section, we provide an application of the general bounds of Section \ref{sec3} to the chi-square approximation of the power divergence statistics. Consider a multinomial goodness-of-fit test over $n$ independent trials, with each trial resulting in a unique classification over $r\geq2$ classes. We denote the observed frequencies arising in the classes by $U_1,\ldots,U_r$, and denote the non-zero classification probabilities by $p_1,\ldots,p_r$. The \emph{power divergence family of statistics}, introduced by \cite{cr84}, are then given by
\begin{equation}\label{pdstat}T_\lambda=\frac{2}{\lambda(\lambda+1)}\sum_{j=1}^rU_j\bigg[\bigg(\frac{U_j}{np_j}\bigg)^\lambda-1\bigg],
\end{equation}
where the index parameter $\lambda\in\mathbb{R}$. When $\lambda=0,-1$, the notation (\ref{pdstat}) should be understood as a result of passage to the limit. The case $\lambda=0$ corresponds to the log-likelihood ratio statistic, whilst other important special cases include the Freeman-Tukey statistic \cite{ft50} ($\lambda=-1/2$) and Pearson's statistic \cite{pearson} ($\lambda=1$):
\begin{equation}
	\chi^2=\sum_{j=1}^{r}\frac{(U_j-np_j)^2}{np_j}.\label{pstat}
\end{equation}

A fundamental result is that, for all $\lambda\in\mathbb{R}$, the statistic $T_\lambda$ converges in distribution to the $\chi_{(r-1)}^2$ distribution as the number of trials $n$ tends to infinity (see \cite{cr84}, p.\ 443). Edgeworth expansions have been used to assess the quality of the chi-square approximation of the distibution of the statistic $T_{\lambda}$ by  \cite{a10,asylbekov,fv20,pu13,read84,ulyanov}. For $r\geq4$ and all $\lambda\in\mathbb{R}$,  \cite{ulyanov} obtained a $O(n^{-(r-1)/r})$ bound on the rate of convergence in the Kolmogorov distance (a refinement of this result is given in \cite{a10}) and, for $r=3$, \cite{asylbekov} obtained a $O(n^{-3/4+0.065})$ bound on the rate of convergence. It has also been shown recently by \cite{pu21} that introducing external randomisation can increase the speed of convergence in the Kolmogorov metric. Special cases of the family have also been considered.  For the likelihood ratio statistic, \cite{ar20} obtained an explicit $O(n^{-1/2})$ bound for smooth test functions (in a setting more general than that of categorical data). For Pearson's statistic, \cite{yarnold} established a $O(n^{-(r-1)/r})$, for $r\geq2$, bound on the rate of convergence in the Kolmogorov distance, which was improved to $O(n^{-1})$ for $r\geq6$ by \cite{gu03}. An explicit $O(n^{-1/2})$ Kolmogorov distance bound was established using Stein's method by \cite{mann97}, whilst \cite{gaunt chi square} used Stein's method to obtain explicit $O(n^{-1})$ bounds, measured using smooth test functions. The results of \cite{gaunt chi square} have recently been generalised to the power divergence statistics by \cite{gaunt pd}, covering the family for $\lambda>-1$, the largest subclass for which finite sample bounds can be obtained.

In this section, we improve the results of \cite{gaunt pd,gaunt chi square} for the case of $r=2$ cell classifications, by obtaining bounds that hold in weaker metrics and possess constants several orders of magnitude smaller. We derive our results by taking advantage of the special structure of Pearson's statistic for the case $r=2$, which allows us to apply the general bounds of Theorem \ref{thm3.2}. The special structure is that the statistic can be written as the square of a sum of i.i.d.\ random variables with zero mean and unit variance; see the proof of Proposition \ref{prop6.1}. 

\begin{proposition}\label{prop6.1} (i) Let $(U_1,U_2)$ represent the vector of $n\geq1$ observed counts, with cell classification probabilities $0<p_1,p_2<1$ . Let $\chi^2$ denote Pearson's statistic, as defined in (\ref{pstat}). 
Then
\begin{equation}\label{p2.41}d_{\mathrm{W}}(\mathcal{L}(\chi^2),\chi_{(1)}^2)\leq \frac{25}{\sqrt{np_1p_2}},
\end{equation}
and, for $h\in C_b^2(\mathbb{R}^+)$,
\begin{equation}\label{p2.42}|\mathbb{E}[h(\chi^2)]-\chi_{(1)}^2h|\leq\frac{892}{np_1p_2}(\|h'\|+\|h''\|),
\end{equation}
where $\chi_{(1)}^2h$ denotes the expectation $\mathbb{E}[h(Y)]$ for $Y\sim\chi_{(1)}^2$.

\vspace{2mm}

\noindent{(ii) More} generally, let $T_\lambda$, $\lambda>-1$, denote the power divergence statistic (\ref{pdstat}). 
Then
\begin{equation}\label{p2.43}d_{\mathrm{W}}(\mathcal{L}(T_\lambda),\chi_{(1)}^2)\leq \frac{1}{\sqrt{np_1p_2}}\bigg(25+\frac{\sqrt{2}|(\lambda-1)(4\lambda+7)|}{\lambda+1}\bigg),
\end{equation}
and, for $h\in C_b^2(\mathbb{R}^+)$,
\begin{align}\label{p2.44}|\mathbb{E}[h(T_\lambda)]-\chi_{(1)}^2h|&\leq\frac{1}{np_1p_2}\bigg\{(892+496|\lambda-1|)(\|h'\|+\|h''\|)+\frac{19}{9}(\lambda-1)^2\|h''\|\nonumber\\
&\quad+\frac{|(\lambda-1)(\lambda-2)(12\lambda+13)|}{6(\lambda+1)}\|h'\|
	\bigg\}.
\end{align}
\end{proposition}

Using the special structure of Pearson's statistic when $r=2$, and applying inequality (\ref{p2.44}) together with a standard smoothing technique for converting smooth test function bounds into Kolmogorov distance bounds, we deduce the following Kolmogorov distance bounds (see \cite[p.\ 48]{chen}, and also \cite{gl22} for a detailed account of the technique). 

\begin{corollary}\label{cor4.4} Suppose $n\geq1$ and $0<p_1,p_2<1.$ Then
	\begin{align}\label{kk1}
		d_{\mathrm{K}}(\mathcal{L}(\chi^2),\chi^2_{(1)})\leq\frac{0.9496}{\sqrt{np_1p_2}}.
	\end{align} 
Let $\lambda>-1$. 
Then, there exists universal constants $C_1,C_2,C_3>0$, independent of $n$, $p_1$, $p_2$ and $\lambda$, such that
\begin{align}\label{kk2}
	d_{\mathrm{K}}(\mathcal{L}(T_\lambda),\chi_{(1)}^2)&\leq \frac{1}{(np_1p_2)^{1/5}}\bigg\{C_1+C_2(\lambda-1)^2+\frac{C_3|(\lambda-1)(\lambda-2)(12\lambda+13)|}{(\lambda+1)(np_1p_2)^{2/5}}
	\bigg\}.
\end{align}	
\end{corollary}

\begin{remark}\label{bdfnd}Order $n^{-1/2}$ bounds for the quantity $|\mathbb{E}[h(T_\lambda)]-\chi_{(r)}^2h|$, $r\geq2$, were obtained by \cite{gaunt pd} for $h\in C_b^2(\mathbb{R}^+)$, whilst $O(n^{-1})$ bounds were obtained for $h\in C_b^5(\mathbb{R}^+)$. These bounds generalised bounds of \cite{gaunt chi square} for the quantity $|\mathbb{E}[h(\chi^2)]-\chi_{(1)}^2h|$, which held for bounded $h$ in the classes $C_b^2(\mathbb{R}^+)$ and $C_b^5(\mathbb{R}^+)$. By obtaining bounds on $|\mathbb{E}[h(T_\lambda)]-\chi_{(1)}^2h|$ that hold for a wider class of functions, in Corollary \ref{cor4.4} we are able to deduce a Kolmogorov distance bound for the $\chi_{(1)}^2$ approximation of $T_\lambda$ for two cell classifications that improves the $O(n^{-1/10})$ bound on the rate of convergence obtained by \cite{gaunt pd} to $O(n^{-1/5})$. This rate of convergence is slower than the optimal $O(n^{-1/2})$ rate for the case $r=2$, but is, to the best knowledge of the author, the only such bound in the literature (except that of \cite{gaunt pd}) that tends to zero under the condition that $np_*\rightarrow\infty$, where $p_*=\mathrm{min}\{p_1,p_2\}$.

Indeed, $np_*\rightarrow\infty$ ($p_*=\mathrm{min}_{1\leq j\leq r}p_j$) is an established condition under which the chi-square approximation of Pearson's statistic is valid \cite{gn96}.
The key to obtaining bounds that tend to zero under this condition was the weaker moment conditions of Theorem \ref{thm3.2}. If any of the absolute moments in the bounds in parts (ii) and (iv) were of larger order, we would not have been able to obtain bounds that tend to zero if $np_*\rightarrow\infty$. As an illustrative example, using Theorem 3.5 of \cite{gaunt normal} instead of part (iv) of Theorem \ref{thm3.2} would have resulted in a $O(n^{-1})$ bound that tends to zero under the stronger condition $n^{5/2}p_*\rightarrow\infty$.
\end{remark}

The following result, which may be of independent interest, is used to prove Proposition \ref{prop6.1}. 

\begin{corollary}\label{cor4.1} Suppose $X_1,\ldots,X_n$ are i.i.d.\ random variables with $\mathbb{E}[X_1]=0$, $\mathbb{E}[X_1^2]=1$ and $\mathbb{E}[X_1^4]<\infty$. Let $W=n^{-1/2}\sum_{i=1}^nX_i$. Then 
\begin{equation}\label{c4.11}d_{\mathrm{W}}(\mathcal{L}(W^2),\chi_{(1)}^2)\leq\frac{1}{\sqrt{n}}\bigg[24\mathbb{E}|X_1|^3+\frac{17}{\sqrt{n}}\mathbb{E}[X_1^4]\bigg].
\end{equation}
Suppose now that $\mathbb{E}[X_1^6]<\infty$. Then, for $h\in C_b^2(\mathbb{R}^+)$,
\begin{align}|\mathbb{E}[h(W^2)]-\chi_{(1)}^2h|&\leq\frac{(\|h'\|+\|h''\|)}{n}\bigg\{187\mathbb{E}[X_1^4]+\frac{131}{n}\mathbb{E}[X_1^6]\nonumber\\
	&\quad+|\mathbb{E}[X_1^3]|\bigg[704\mathbb{E}|X_1|^3+\frac{468}{n}\mathbb{E}|X_1|^5\bigg]\bigg\}.\label{c4.12}
\end{align}
\end{corollary}

\begin{remark}The bound (\ref{c4.12}) improves on the bound of Theorem 3.1 of \cite{gaunt chi square} for the quantity $|\mathbb{E}[h(W^2)]-\chi_{(1)}^2h|$ by holding for a larger class of test functions, having weaker moment conditions, and also possessing smaller numerical constants. 
\end{remark}

\begin{proof}To prove inequality (\ref{c4.11}), we apply part (ii) of Theorem \ref{thm3.2}. 
We have $g(w)=w^2$ and, since $g'(w)=2w$, we take $P(w)=2|w|$ as our dominating function.
Applying inequality (\ref{thm3.3b}) (taking note of Remark \ref{vdfeb} to get a Wasserstein distance bound) with $A=0, B=2, r=1$ and $p=2$, and using that $\mathbb{E}|W|\leq(\mathbb{E}[W^2])^{1/2}=1$ and $\mu_1=\sqrt{2/\pi}$, we obtain the bound (\ref{c4.11}), after rounding numerical constants up to the nearest integer.


We now prove inequality (\ref{c4.12}). Since $g(w)=w^2$ is an even function, we can apply part (iv) of Theorem \ref{thm3.2} to obtain a $O(n^{-1})$ bound. We have $(g'(w))^2=4w^2$ and $g''(w)=2$, and so take $P(w)=2+4w^2$ as our dominating function. Applying inequality (\ref{thm3.5b}) with $A=2, B=4, r=2$ and $p=2,$ and using that $\mathbb{E}[W^2]=1$ and $\mu_3=2\sqrt{2/\pi}$, now yields inequality (\ref{c4.12}).
\end{proof}

We will use the following lemma in our proof of Proposition \ref{prop6.1}. The proof, which involves an application of Theorem \ref{thm3.2}, is deferred until the end of the section.

\begin{lemma}\label{steinlem2}Let $X_1,\ldots,X_n$ be i.i.d$.$ random variables with $X_1=_d(I-p_1)/\sqrt{p_1p_2}$, where $I\sim \mathrm{Ber}(p_1)$. Let $W=n^{-1/2}\sum_{i=1}^nX_i$. Then, for $h\in C_b^2(\mathbb{R}^+)$,
\begin{align}\label{g1g1bound}|\mathbb{E}[W^3h'(W^2)]|\leq\frac{2976}{\sqrt{np_1p_2}}\{\|h'\|+\|h''\|\}.
\end{align}
\end{lemma}

\noindent{\emph{Proof of Proposition \ref{prop6.1}.}} (i)
We begin by proving inequalities (\ref{p2.41}) and (\ref{p2.42}).
 Let $p=p_1$, so that $p_2=1-p$. As $U_1+U_2=n$, a short calculation gives that 
\begin{equation*}\chi^2=\bigg(\frac{U_1-np}{\sqrt{np(1-p)}}\bigg)^2.
\end{equation*}
Since $U_1\sim \mathrm{Bin}(n,p)$,  it can be expressed as a sum of i.i.d.\ indicator random variables: $U_1=\sum_{i=1}^nI_i$, where $I_i\sim\mathrm{Ber}(p)$. We can therefore write $\chi^2=W^2$, where $W=n^{-1/2}\sum_{i=1}^n X_i$, for $X_1,\ldots, X_n$ i.i.d.\ random variables with $X_1=(I_1-p)/\sqrt{p(1-p)}$. We have that $\mathbb{E}[X_1]=0$, $\mathbb{E}[X_1^2]=1$ and $|\mathbb{E}[X_1^m]|\leq\mathbb{E}|X_1|^m\leq (p(1-p))^{1-m/2}$ for all $m\geq2$. The assumptions for inequalities (\ref{c4.11}) and (\ref{c4.12}) therefore hold. 
Plugging these moment bounds into (\ref{c4.11}) and rounding  constants up to the nearest integer then yields the bound
\begin{equation}\label{ssll}d_{\mathrm{W}}(\mathcal{L}(\chi^2),\chi_{(1)}^2)\leq\frac{1}{\sqrt{np_1p_2}}\bigg(24+\frac{17}{\sqrt{np_1p_2}}\bigg).
\end{equation}
We obtain the simplified bound (\ref{p2.41}) as follows. Let $Y\sim\chi_{(1)}^2$. Observe that, for $h\in\mathcal{H}_{\mathrm{W}}$,
\begin{align*}
	|\mathbb{E}[h(\chi^2)]-\chi^2_{(1)}h|\leq\|h'\|\mathbb{E}|\chi^2-Y|\leq\|h'\|\big(\mathbb{E}[\chi^2]+\mathbb{E}[Y]\big)=2\|h'\|\leq2,
\end{align*}
and so 
\begin{equation}\label{p4.2ineq1}d_{\mathrm{W}}(\mathcal{L}(\chi^2),\chi_{(1)}^2)\leq2.
\end{equation}
 Now, the upper bound in (\ref{ssll}) is greater than the upper bound in (\ref{p4.2ineq1}) if $\sqrt{np_1p_2}<17$, and so we may take $\sqrt{np_1p_2}\geq17$ in (\ref{ssll}), and doing so yields the bound (\ref{p2.41}). Similarly, we obtain inequality (\ref{p2.42}) by plugging the moment bounds into (\ref{c4.12}) and simplifying the bound as we did in deriving inequality (\ref{p2.41}).

\vspace{2mm}

\noindent{(ii)} To prove inequality (\ref{p2.43}), we use the following bound, which can be found by examining the proof of Theorem 2.3 of \cite{gaunt pd}:
\begin{equation}\label{polm}
	|\mathbb{E}[h(T_\lambda)]-\chi_{(1)}^2h|\leq|\mathbb{E}[h(\chi^2)]-\chi_{(1)}^2h|+\frac{|(\lambda-1)(4\lambda+7)|\|h'\|}{(\lambda+1)\sqrt{n}}\bigg(\frac{1}{\sqrt{p_1}}+\frac{1}{\sqrt{p_2}}\bigg).
\end{equation}
Using inequality (\ref{p2.41}) to bound $|\mathbb{E}[h(\chi^2)]-\chi_{(1)}^2h|$ and also using the inequality $1/\sqrt{p_1}+1/\sqrt{p_2}\leq\sqrt{2/(p_1p_2)}$ now yields the desired bound. We note that inequality (\ref{polm}) was derived under the assumption that $np_*\geq1$, and that if $\lambda\geq2$ then also $np_*\geq2(\lambda-2)^2$, where $p_*=\mathrm{min}\{p_1,p_2\}$. However, if these assumptions are not satisfied then the bound in (\ref{p2.43}) exceeds 2 (the upper bound in (\ref{p4.2ineq1})), and so there is no need to impose these conditions in the statement of the theorem; similar comments apply to inequality (\ref{pdrr}) below.

Finally, we prove inequality (\ref{p2.44}). We use the following bound, which can be found by examining the proof of Theorem 2.2 of \cite{gaunt pd}:
\begin{align}
	|\mathbb{E}[h(T_\lambda)]-\chi_{(1)}^2h|&\leq \frac{|(\lambda-1)(\lambda-2)(12\lambda+13)|}{6(\lambda+1)n}\bigg(\frac{1}{p_1}+\frac{1}{p_2}\bigg)\|h'\|\nonumber\\
\label{pdrr}	&\quad+\frac{19(\lambda-1)^2}{18n}\bigg(\frac{1}{\sqrt{p_1}}+\frac{1}{\sqrt{p_2}}\bigg)^2\|h''\|+|\mathbb{E}[h(\chi^2)]-\chi_{(1)}^2h|+R,
\end{align}
where
\begin{align*}
	R=\frac{|\lambda-1|}{3}\bigg|\mathbb{E}\bigg[\bigg(\frac{S_1^3}{\sqrt{np_1}}+\frac{S_2^3}{\sqrt{np_2}}\bigg)h'(W^2)\bigg]\bigg|,
\end{align*}
and $S_j=(U_j-np_j)/\sqrt{np_j}$, $j=1,2$.
To bound $R$, we express $S_1$ and $S_2$ in terms of $W$. Using the representation of $W^2=\chi^2$ given at the beginning of the proof and the alternative representation $W=(U_2-np_2)/{\sqrt{np_1p_2}}$,
which is obtained using a similar calculation, we have the representations $S_1=\sqrt{p_2}W$ and $S_2=\sqrt{p_1}W. $ Therefore
\begin{align}
R&=\frac{|\lambda-1|}{3}\bigg(\frac{p_2^{3/2}}{\sqrt{np_1}}+\frac{p_1^{3/2}}{\sqrt{np_2}}\bigg)|\mathbb{E}[W^3h'(W^2)]|\nonumber\\
\label{pdrr2}&	\leq\frac{992|\lambda-1|}{np_1p_2}\big(p_2^{3/2}p_1^{1/2}+p_1^{3/2}p_2^{1/2}\big)\big(\|h'\|+\|h''\|\big)\leq\frac{496|\lambda-1|}{np_1p_2}\big(\|h'\|+\|h''\|\big),
\end{align}
where we used Lemma \ref{steinlem2} to obtain the first inequality. We now obtain inequality (\ref{p2.44}) by substituting inequality (\ref{pdrr2}) and our bound (\ref{p2.42}) for $|\mathbb{E}[h(\chi^2)]-\chi_{(1)}^2h|$ into (\ref{pdrr}), and simplifying the resulting bound using the formula $1/p_1+1/p_2= 1/(p_1p_2)$ and the inequality $(1/\sqrt{p_1}+1/\sqrt{p_2})^2\leq2/(p_1p_2)$, for $0<p_1,p_2<1$ such that $p_1+p_2=1$.
\qed

\vspace{3mm}

\noindent{\emph{Proof of Corollary \ref{cor4.4}.}	We first prove inequality (\ref{kk1}). Recall from the proof of Proposition \ref{prop6.1} that when $r=2$ we can write Pearson's statistic as $\chi^2=W^2$, where $W=n^{-1/2}\sum_{i=1}^{n}X_{i}$ for $X_{1},\ldots,X_{n}$ i.i.d.\ random variables with $X_1=_d(I_1-p)/\sqrt{p(1-p)}$ and $I_1\sim \mathrm{Ber}(p)$. Recall also that if $Z\sim N(0,1)$ then $Z^2\sim\chi_{(1)}^2$. Thus, for any $z>0$, 
	\begin{align*}
		|\mathbb{P}(W^2\leq z)-\mathbb{P}(Z^2\leq z)|&=|\mathbb{P}(-\sqrt{z}\leq W^2\leq \sqrt{z})-\mathbb{P}(-\sqrt{z}\leq Z^2\leq \sqrt{z})|\\
		&\leq|\mathbb{P}(W\leq\sqrt{z})-\mathbb{P}(Z\leq z)|+|\mathbb{P}(W\leq -\sqrt{z})-\mathbb{P}(Z\leq -\sqrt{z})|,
	\end{align*}
from which we deduce that $	d_{\mathrm{K}}(\mathcal{L}(\chi^2),\chi^2_{(1)})\leq2d_{\mathrm{K}}(\mathcal{L}(W),\mathcal{L}(Z))$.
Now, using the Berry-Esseen theorem with the best available numerical constant of $C=0.4748$ \cite{s11}, we get:
\begin{align*}
	d_{\mathrm{K}}(\mathcal{L}(\chi^2),\chi^2_{(1)})\leq2\frac{0.4748\mathbb{E}|X_1^3|}{(\mathbb{E}[X_1^2])^{3/2}\sqrt{n}}=\frac{0.9496}{\sqrt{np_1p_2}}.
\end{align*}

We now prove inequality (\ref{kk2}). Let $\alpha>0$, and for fixed $z>0$ define a function $h_\alpha:\mathbb{R}^+\rightarrow[0,1]$ by $h_\alpha(x)=1$ if $x\leq z$; $h_\alpha(x)=1-2(x-z)^2/\alpha^2$ if $z<x\leq z+\alpha/2$; $h_\alpha(x)=2(x-(z+\alpha))^2/\alpha^2$ if $z+\alpha/2<x\leq z+\alpha$; and $h_\alpha(x)=0$ if $x\geq z+\alpha$. 
Then $h_\alpha'$ is Lipschitz with $\|h_\alpha'\|=2/\alpha$ and $\|h_\alpha''\|=4/\alpha^2$.  Let $Y\sim\chi_{(1)}^2$. Using inequality (\ref{p2.44}) now yields
\begin{align}&\mathbb{P}(T_\lambda\leq z)-\mathbb{P}(Y\leq z)\nonumber\\
&\quad\leq \mathbb{E}[h_\alpha(T_\lambda)]-\mathbb{E}[h_\alpha(Y)]+\mathbb{E}[h_\alpha(Y)]-\mathbb{P}(Y\leq z)\nonumber \\
&\quad\leq \frac{1}{np_1p_2}\bigg\{(892+496|\lambda-1|)\bigg(\frac{2}{\alpha}+\frac{4}{\alpha^2}\bigg)+\frac{19}{9}(\lambda-1)^2\cdot\frac{4}{\alpha^2}\nonumber\\
\label{kol123}&\quad\quad+\frac{|(\lambda-1)(\lambda-2)(12\lambda+13)|}{6(\lambda+1)}\cdot\frac{2}{\alpha}
	\bigg\}+\mathbb{P}(z\leq Y\leq z+\alpha).
\end{align}
We note the inequality $\mathbb{P}(z\leq Y\leq z+\alpha)\leq\sqrt{2\alpha/\pi}$ (see \cite[p.\ 754]{gaunt chi square}). An upper bound now follows on applying this inequality to (\ref{kol123}) and choosing $\alpha=
c(np_1p_2)^{-2/5}$, for some universal constant $c>0$. To simplify the bound, we used that basic inequality $|\lambda-1|\leq 1+(\lambda-1)^2$.  Similarly, we can obtain a lower bound that is the negative of the upper bound, which completes the proof. 
\qed

\vspace{3mm}

\noindent{\emph{Proof of Lemma \ref{steinlem2}.} Let $g(w)=w^3h'(w^2)$. We begin by noting that $\mathbb{E}[g(Z)]=\mathbb{E}[Z^3h'(Z^2)]=0$. This is because the standard normal distribution is symmetric about the origin ($Z=_d -Z$, for $Z\sim N(0,1)$) and $g(w)$ is an odd function ($g(w)=-g(-w)$ for all $w\in\mathbb{R}$), meaning that $\mathbb{E}[g(Z)]=\mathbb{E}[g(-Z)]=-\mathbb{E}[g(Z)]$, whence $\mathbb{E}[g(Z)]=0$. Therefore, we can write $\mathbb{E}[W^3h'(W^2)]=\mathbb{E}[W^3h'(W^2)]-\mathbb{E}[Z^3h'(Z^2)]$.  

We shall obtain the bound (\ref{g1g1bound}) by applying part (ii) of Theorem \ref{thm3.2} with $h(w)=w$ and $g(w)=w^3h'(w^2)$. Using the basic inequality $2a^2\leq 1+a^4$, we have, for $w\in\mathbb{R}$,
\begin{equation*} |g'(w)|=|2w^4h''(w^2)+3w^2h'(w^2)|\leq 2w^4\|h''\|+3w^2\|h'\|\leq\frac{3}{2}\|h'\|+\Big(2\|h''\|+\frac{3}{2}\|h'\|\Big)w^4.
\end{equation*}
We therefore can apply part (ii) of Theorem \ref{thm3.2} with $A=3\|h'\|/2$, $B=2\|h''\|+3\|h'\|/2$, $r=4$ and $p=2$, and using the bound (\ref{thm3.3b}) gives that
\begin{align*}&|\mathbb{E}[W^3h'(W^2)]|=|\mathbb{E}[W^3h'(W^2)]-\mathbb{E}[Z^3h'(Z^2)]| \\
&\leq \frac{3}{2\sqrt{n}}\bigg[\frac{21}{2}\|h'\|\mathbb{E}|X_1^3|+4\Big(2\|h''\|+\frac{3}{2}\|h'\|\Big)\bigg(72\mathbb{E}|X_1^3|\mathbb{E}[W^4]+\frac{72}{n^{2}}\mathbb{E}|X_1^{7}|+40\sqrt{\frac{2}{\pi}}\mathbb{E}|X_1^3|\bigg)\bigg]\\
&\leq\frac{1}{\sqrt{np_1p_2}}\bigg[\|h'\|\bigg(2247+\frac{648}{np_1p_2}+\frac{648}{(np_1p_2)^2}\bigg)+\|h''\|\bigg(2975+\frac{864}{np_1p_2}+\frac{864}{(np_1p_2)^2}\bigg)\bigg]\\
&\leq \frac{1}{\sqrt{np_1p_2}}\{\|h'\|+\|h''\|\}\bigg(2975+\frac{864}{np_1p_2}+\frac{864}{(np_1p_2)^2}\bigg),
\end{align*}
where we used that $\mu_5=8\sqrt{2/\pi}$, and that $\mathbb{E}|X_1^3|\leq(p_1p_2)^{-1/2}$, $\mathbb{E}|X_1^7|\leq(p_1p_2)^{-5/2}$ and $\mathbb{E}[W^4]=3(n-1)(\mathbb{E}[X_1^2])^2/n+\mathbb{E}[X_1^4]/n\leq 3+1/(np_1p_2)$. We obtain the final simplified bound (\ref{g1g1bound}) using a similar argument to the one used in the proof of Proposition \ref{prop6.1}.
\qed

\appendix

\section{Further proofs}\label{appendix}

\noindent{\emph{Proof of Lemma \ref{lem2}.}} The case $r=0$ is trivial. Let us prove inequality (\ref{abineq}) for integer $r\geq1$. By the binomial theorem, 
\begin{align*}
	(ax+by)^r=\sum_{k=0}^{r}\binom{r}{k}(ax)^k(by)^{r-k}\leq\sum_{k=0}^{r}\binom{r}{k}a^kb^{r-k}(x^r+y^r)=(a+b)^r(x^r+y^r),
\end{align*}
where the inequality follows from an application of the basic inequality $x^ky^{r-k}\leq x^r+y^r$.
The proof for general $r>0$ is exactly the same except that we must instead use the generalised binomial theorem. The proof of  inequality (\ref{abcineq}) is also very similar, with the same basic argument, but instead we apply the generalised trinomial theorem and the basic inequality $x^iy^jz^k\leq x^r+y^r+z^r$ for $i,j,k\geq0$ such that $i+j+k=r$.
\qed
\vspace{3mm}

\noindent{\emph{Proof of Lemma \ref{lem3}.}} Making the change of variables $u=t^2/2$ allows us to write $T_r(w)$ in terms of the upper incomplete gamma function $\Gamma(a,x)=\int_{x}^{\infty}u^{a-1}\mathrm{e}^{-u}\,\mathrm{d}u$:
\begin{equation}\label{Trineq}
	T_r(w)=2^{(r-1)/2}w\mathrm{e}^{w^2/2}\Gamma\bigg(\frac{r+1}{2}, \frac{w^2}{2}\bigg).
\end{equation}
We now note two upper bounds for the incomplete gamma function. If $0<a\leq1,$ then $\Gamma(a,x)\leq x^{a-1}\mathrm{e}^{-x}$ for all $x>0$ (See \cite{jameson16}), whilst if $a>1$, then $\Gamma(a,x)\leq Bx^{a-1}\mathrm{e}^{-x}$, for $x>\frac{B}{B-1}(a-1), B>1$ (see \cite{np00}). For $0\leq r\leq 1$, we apply the bound of \cite{jameson16} to (\ref{Trineq}) to obtain inequality (\ref{Tr1}). To obtain inequality (\ref{Tr2}), we instead apply the inequality of \cite{np00} with $B=2$ to (\ref{Trineq}). \qed
		
\section*{Acknowledgements}
HS is supported by an EPSRC PhD Studentship.


\footnotesize
		
		\begin{thebibliography}{99}
			\addcontentsline{toc}{section}{References}

\bibitem{ag20} Anastasiou, A. and Gaunt, R. E. Multivariate normal approximation of the maximum likelihood estimator via the delta method. \emph{Braz. J. Probab. Stat.} $\mathbf{34}$ (2020), 136--149.

\bibitem{ar20} Anastasiou, A. and Reinert, G. Bounds for the asymptotic distribution of the likelihood ratio. \emph{Ann. Appl. Probab.} $\mathbf{30}$ (2020),  608--643.
			
\bibitem{a10} Assylebekov, Z. A. Convergence Rate of Multinomial Goodness-of-fit Statistics to Chi-square Distribution. \emph{Hiroshima Math. J.} $\mathbf{40}$ (2010), 115--131.

\bibitem{asylbekov} Assylebekov, Z. A., Zubov, V. N. and Ulyanov V. V.  On approximating some statistics of goodness-of-fit tests in the case of three-dimensional discrete data. \emph{Siberian Math. J.} $\mathbf{52}$ (2011),  571--584.			
		
\bibitem{barbour2} Barbour, A. D. Stein's method for diffusion approximations.  \emph{Probab. Theory Rel.} $\mathbf{84}$ (1990),  297--322.

\bibitem{bh85} Barbour, A. D. and Hall, P. On bounds to the rate of convergence in the central limit theorem. \emph{Bull. Lond. Math. Soc.} $\mathbf{17}$ (1985), 151--156. 

\bibitem{bt10} Berend, D. and Tassa, T.  Improved bounds on Bell numbers and on moments of sums of random variables. \emph{Probab.  Math. Stat.} $\mathbf{30}$ (2010), 185--205.

\bibitem{chen} Chen, L. H. Y., Goldstein, L. and Shao, Q.--M.  \emph{Normal Approximation by Stein's Method.} Springer, 2011.

\bibitem{cr84} Cressie, N. and Read, T. R. C.  Multinomial Goodness-of-Fit Tests. \emph{J. Roy. Stat. Soc. B Met.}  $\mathbf{46}$ (1984),  440--464.


\bibitem{dgv17} D\"{o}bler, C, Gaunt, R. E. and Vollmer, S. J.  An iterative technique for bounding derivatives of solutions of Stein equations.  \emph{Electron. J. Probab.} $\mathbf{22}$ no. 96 (2017).


\bibitem{elezovic} Elezovi\'c, N., Giordano, C. and Pe\v{c}ari\'c, J.  The best bounds in Gautschi's inequality. \emph{Math. Inequal. Appl.} $\mathbf{3}$ (2000),  239--252.

\bibitem{fgrs22} Fischer, A., Gaunt, R. E., Reinert, G. and Swan, Y. Normal approximation for the posterior in exponential families. arXiv:2209.08806, 2022.

\bibitem{ft50} Freeman, M. F. and Tukey, J. W. Transformations Related to the Angular and the Square Root. \emph{Ann. Math. Statist.} $\mathbf{21}$ (1950),  607--611.

\bibitem{fv20} Fujikoshi, Y. and Ulyanov, V. V. \emph{Non-Asymptotic Analysis of Approximations for Multivariate Statistics.} Springer Briefs, 2020.

\bibitem{gaunt inequality} Gaunt, R. E.  Inequalities for modified Bessel functions and their integrals.  \emph{J. Math. Anal. Appl.} $\mathbf{420}$ (2014), 373--386. 

\bibitem{gaunt rate} Gaunt, R. E. Rates of Convergence in Normal Approximation Under Moment Conditions Via New Bounds on Solutions of the Stein Equation.  \emph{J. Theoret. Probab.} $\mathbf{29}$ (2016),  231--247.

\bibitem{gaunt normal} Gaunt, R. E. Stein's method for functions of multivariate normal random variables. \emph{Ann. I. H. Poincare-PR} $\mathbf{56}$ (2020), 1484--1513. 

\bibitem{gaunt pd} Gaunt, R. E.  Bounds for the chi-square approximation of the power divergence family of statistics. To appear in \emph{J. Appl. Probab.}, 2022+.

\bibitem{gl22} Gaunt, R. E. and Li, S. Bounding Kolmogorov distances through Wasserstein and related integral probability metrics. arXiv:2201.12087, 2022.

\bibitem{gaunt chi square} Gaunt, R. E., Pickett, A. and Reinert, G.  Chi-square approximation by Stein's method with application to Pearson's statistic. \emph{Ann. Appl. Probab.} $\mathbf{27}$ (2017),  720--756. 

\bibitem{gr21} Gaunt, R. E. and Reinert, G. Bounds for the chi-square approximation of Friedman's statistic by Stein's method. To appear in \emph{Bernoulli}, 2022+.

	
		\bibitem{goldstein1} Goldstein, L. and Rinott, Y.  Multivariate normal approximations by Stein's method and size			bias couplings.  \emph{J. Appl. Probab.} $\mathbf{33}$ (1996),   1--17.
			
		\bibitem{gotze} G\"{o}tze, F.  On the rate of convergence in the multivariate CLT.  \emph{Ann. Probab.} $\mathbf{19}$ (1991),  724--739.
		
\bibitem{gu03} G\"{o}tze, F. and Ulyanov, V. V.  Asymptotic distribution of $\chi^2$-type statistics. Preprint 03--033 Research group Spectral analysis, asymptotic distributions and stochastic dynamics, Bielefeld Univ., Bielefeld, 2003.		
		
\bibitem{gn96} Greenwood, P. E. and Nikulin, M. S. \emph{A Guide to Chi-Squared Testing.} Wiley, New York, 1996.

\bibitem{ismail} Ismail, M. E. H., Lorch, L. and Muldoon, M. E. Completely monotonic functions associated with the gamma function and its q-analogues. \emph{J. Math. Anal. Appl.} $\mathbf{116}$ (1986),  1--9.
		
\bibitem{jameson16} Jameson, G. J. O. The incomplete gamma functions. \emph{Math. Gazette} $\mathbf{100}$ (2016), 298--306.		
\bibitem{mann97} Mann, B.  Convergence rate for a $\chi^2$ of a multinomial. \emph{Unpublished manuscript}, 1997.
		
\bibitem{mz37} Marcinkiewicz, J. and Zygmund, A. Sur les fonctions ind\'ependantes. \emph{Fundam. Math.} $\mathbf{29}$ (1937),  60--90.

\bibitem{np00} Natalini, P. and Palumbo, B. Inequalities for the incomplete gamma function. \emph{Math. Inequal. Appl.} $\mathbf{3}$ (2000), 69--77.
	
\bibitem{olver} Olver, F. W. J., Lozier, D. W., Boisvert, R. F. and Clark, C. W.  \emph{NIST Handbook of Mathematical Functions.} Cambridge University Press, 2010.

\bibitem{pearson} Pearson, K.  On the criterion that a given system of deviations is such that it can be reasonably supposed to have arisen from random sampling.  \emph{Phil. Mag.} $\mathbf{50}$ (1900),  157--175.

\bibitem{pickett} Pickett, A. \emph{Rates of Convergence of $\chi^2$ Approximations via Stein's Method.}  DPhil thesis, University of Oxford, 2004.


\bibitem{pu13} Prokhorov, Y. V. and Ulyanov, V. V. Some Approximation Problems in Statistics and Probability. Eichelsbacher, P. et al. (eds.), \emph{Limit Theorems in Probability, Statistics and Number Theory}, Springer Proceedings in Mathematics \& Statistics $\mathbf{42}$ (2013),  235--249.

\bibitem{pu21} Puchkin, N. and Ulyanov, V. Inference via Randomized Test Statistics. To appear in \emph{Ann. I. H. Poincare-PR}, 2022+.

\bibitem{read84} Read, T. R. C. Closer asymptotic approximations for the distributions of the power divergence goodness-of-fit statistics. \emph{Ann. I. Stat. Math.}  $\mathbf{36}$ (1984),  59--69.

\bibitem{reinert 0}  Reinert, G.  Three general approaches to Stein's method. In \emph{An Introduction
to Stein's Method. Lect. Notes Ser. Inst. Math. Sci. Natl. Univ. Singap.} Eds: Barbour, A. D. and Chen L. H. Y., $\mathbf{4}$ (2005), pp. 183--221.
Singapore Univ. Press, Singapore.


\bibitem{s11} Shevtsova, I. On the absolute constants in the Berry Esseen type inequalities for identically distributed summands. arXiv:1111.6554, 2011.
			
\bibitem{stein} Stein, C.  A bound for the error in the normal approximation to the the distribution of a sum of dependent random variables.  In \emph{Proc. Sixth Berkeley Symp. Math. Statist. Prob.} (1972), vol. 2, Univ. California Press, Berkeley,  583--602.

\bibitem{ulyanov} Ulyanov, V. V. and Zubov, V. N. Refinement on the convergence of one family of goodness-of-fit statistics to chi-squared distribution. \emph{Hiroshima Math. J.} $\mathbf{39}$ (2009),  133--161.

\bibitem{yarnold} Yarnold, J. K.  Asymptotic approximations for the probability that a sum of lattice random vectors lies in a convex set. \emph{Ann. Math. Statist.} $\mathbf{43}$, (1972),  1566--1580.

			
\end{thebibliography}
\end{document}